\documentclass[11pt]{article}
\usepackage[utf8]{inputenc}
\usepackage{latexsym,amssymb,amsmath,graphics,cite,arydshln}
\usepackage{tikz}
\usepackage{fancyhdr}
\usepackage{lipsum}
\usepackage{lineno}
\usepackage{color}
\usepackage{makecell}
\usepackage{amsfonts}
\usepackage{multirow} %multirow for format of table  
\usepackage{xcolor}
\usepackage{colortbl}
\usepackage{booktabs}
\usepackage{diagbox}

\usepackage[colorlinks,
linkcolor=blue,       %%修改此处为你想要的颜色
anchorcolor=blue,  %%修改此处为你想要的颜色
citecolor=blue,        %%修改此处为你想要的颜色，例如修改blue为red
]{hyperref}

\newtheorem{Theorem}{Theorem}[section]

\newtheorem{Corollary}[Theorem]{Corollary}
\newtheorem{Lemma}[Theorem]{Lemma}

\newtheorem{Remark}[Theorem]{Remark}
\newcommand{\qed}{\hphantom{.}\hfill $\Box$\medbreak}

\begin{document}
\title{On $\ell$-weakly cross $t$-intersecting families for sets and vector spaces\footnote{(Corresponding author: Lijun Ji)}}

\author{\small  Shuhui \ Yu, Lijun\ Ji \\	\small Department of Mathematics and Physics, Suzhou Polytechnic University, Suzhou 215104, China\\ \small Department of Mathematics, Soochow University,	Suzhou 215006, China\\	\small E-mail address: yushuhui\_suda@163.com, jilijun@suda.edu.cn}
\date{}
\maketitle
\begin{abstract}
Let $[n]$ (resp. $V$) be an $n$-element set (resp. $n$-dimensional vector space over the finite field $\mathbb{F}_{q}$), and $\binom{[n]}{k}$ (resp. $\genfrac{[}{]}{0pt}{}{V}{k}$) denote the set of all $k$-subsets of $[n]$ (resp. $k$-dimensional subspaces of $V$). We say that $\mathcal{F}\subseteq\binom{[n]}{k}$ (resp. $\mathcal{F}\subseteq \genfrac{[}{]}{0pt}{}{V}{k}$) and $\mathcal{G}\subseteq \binom{[n]}{k'}$ (resp. $\mathcal{G}\subseteq \genfrac{[}{]}{0pt}{}{V}{k'}$) are $\ell$-weakly cross $t$-intersecting if $\sum_{1\leq i,j\leq \ell}|F_{i}\cap G_{j}|\geq \ell^{2}t-\ell+1$ (resp. $\sum_{1\leq i,j\leq \ell}\dim(F_{i}\cap G_{j})\geq \ell^{2}t-\ell+1$) for all distinct $F_{1},\ldots,F_{\ell}\in\mathcal{F}$ and $G_{1},\ldots,G_{\ell}\in\mathcal{G}$. In this paper, we provide an alternative proof of the set version of the $\ell$-weakly cross $t$-intersecting theorem and an explicit lower bound for $n$. Moreover, we prove that if $\mathcal{F}$ and $\mathcal{G}$ are $\ell$-weakly cross $t$-intersecting subspace families, then
\[
|\mathcal{F}| \cdot |\mathcal{G}| \leq\genfrac{[}{]}{0pt}{}{n-t}{k-t}\genfrac{[}{]}{0pt}{}{n-t}{k'-t}
\]
holds, provided that $n\geq (2k-t+1)(t+1)+(k-t+1)k'+k+2\ell-1$. This extends the theorem of Cao, Lu, Lv and Wang [J. Combin. Theory Ser. A 193 (2023), 105688], who established the upper bound for the product of the sizes of cross $t$-intersecting subspace families.

 \smallskip
\end{abstract}

\section{Introduction}

The study of intersecting families occupies a central role in extremal set theory, originating from the following classical Erdős–Ko–Rado theorem \cite{erdos1961}. 
\begin{Theorem}[\cite{erdos1961}]\label{set-EKR}
Let $n$, $k$ and $t$ be positive integers satisfying $n > k > t$, and let $\mathcal{F} \subseteq \binom{[n]}{k}$ be a family of $k$-subsets of $[n] = \{1,2,\dots,n\}$ such that $|F \cap F'| \geq t$ for all $F, F' \in \mathcal{F}$. Then the following bounds hold.

\begin{itemize}
\item[(i)] If $t = 1$ and $n \geq 2k$, then $|\mathcal{F}| \leq \binom{n-1}{k-1}.$

\item[(ii)] If $t \geq 2$ and $n \geq n_{0}(k,t)$, then $|\mathcal{F}| \leq \binom{n-t}{k-t}.$

\end{itemize}
\end{Theorem}

It is known that the smallest value of $n_{0}(k,t)$ equals $(k-t+1)(t+1)$, which was proved by Frankl \cite{frankl1978} for $t \geq 15$ and subsequently by Wilson\cite{wilson1984} for all $t$.

This fundamental result has various generalizations, among which cross-intersecting families have received considerable attention. For a given positive integer $t$, we say $\mathcal{F}\subseteq \binom{[n]}{k}$ and $\mathcal{G}\subseteq \binom{[n]}{k'}$ are cross $t$-intersecting if $|F\cap G|\geq t$ holds for all $F\in\mathcal{F}$ and $G\in\mathcal{G}$. When $t=1$, cross $1$-intersecting is also referred to as cross intersecting. In \cite{pyber1986}, Pyber generalized  Erdős–Ko–Rado theorem to cross-intersecting setting by establishing the following results.

\begin{Theorem}[\cite{pyber1986}]\label{CI}
Let $n$, $k$ and $k'$ be positive integers satisfying $n > k\geq k'$. Suppose that  $\mathcal{F} \subseteq \binom{[n]}{k}$ and $\mathcal{G} \subseteq \binom{[n]}{k'}$ are cross intersecting families. Then 

\begin{itemize}
\item[(i)] If $k=k'$ and $n \geq 2k$, then $|\mathcal{F}||\mathcal{G}|  \leq \binom{n-1}{k-1}^{2}.$

\item[(ii)] If $k>k'$ and $n \geq 2k+k'-2$, then $|\mathcal{F}||\mathcal{G}|  \leq \binom{n-1}{k-1}\binom{n-1}{k'-1}.$

\end{itemize}
\end{Theorem}
When $k>k'$, the lower bound on $n$ is not sharp. In \cite{MT1989}, Matsumoto and Tokushige improved the bound to $n\geq 2k$.
For $t\geq 2$, Tokushige \cite{T2013} obtained a similar upper bound for $\frac{k}{n}<1-\frac{1}{\sqrt[t]{2}}$ using the eigenvalue method. In the same paper \cite{T2013}, Tokushige conjectured that $|\mathcal{F}||\mathcal{G}|  \leq \binom{n-t}{k-t}\binom{n-t}{k'-t}$ holds when $n\geq (t+1)(\max\{k,k'\}-t+1)$. When $k=k'$,  the conjecture was first verified by  Frankl et al. \cite{FLST2014} for $t\geq 14$ and $n\geq (t+1)k$. Subsequently, Zhang and Wu \cite{ZW2025} solved the conjecture for $t\geq 3$ using the shift operator and generating set method. The remaining case $t=2$ was resolved by Tanaka and Tokushige \cite{TT2025} using the semidefinite programming method. When $k\neq k'$, Borg \cite{Borg2014,Borg2016} proved the conjecture for large $n$. Recently, He et al. \cite{HLWZ2025} proved the conjecture under the condition $t\geq 3$, Chen et al. \cite{CLWZ2025} confirmed the conjecture for $t=2$ and $n\geq 3.38k$.

 In recent years, there has been growing interest in intersecting families with various intersection conditions (for instance, almost intersecting families are considered in \cite{FK2021}). This direction is to investigate weaker intersection conditions that still imply the same upper bound, such as  the sum-type intersection condition on collections of sets \cite{nagy2025,frankl2026,KW2026}. In particular, Nagy \cite{nagy2025} (for the case $\ell=3$) and Frankl et al. \cite{frankl2026} obtained the following theorem.

\begin{Theorem}[\cite{frankl2026}]
Let $n$ and $k$ be positive integers. Suppose that $\mathcal{F}\subseteq \binom{[n]}{k}$ satisfies
\[
\sum_{1\leq i<j\leq \ell}|F_{i}\cap F_{j}|\geq \binom{\ell-1}{2}+1
\]
for every collection of $\ell$ distinct sets $F_{1},\dots,F_{\ell}\in\mathcal{F}$. If $n$ is sufficiently large, then $|\mathcal{F}|\leq \binom{n-1}{k-1}$. Moreover, the threshold $\binom{\ell-1}{2}+1$ is sharp.
\end{Theorem}

Recently, Ai et al. \cite{ai2026} weakened the condition of cross $t$-intersection. Let $n$, $k$ and $k'$ be positive integers and let $\mathcal{F}\subseteq \binom{[n]}{k}$ and $\mathcal{G}\subseteq \binom{[n]}{k'}$. Given positive integers $\ell$ and $t$,  if $$\sum_{1\leq i,j\leq \ell}|F_{i}\cap G_{j}|\geq \ell^2 t-\ell+1 $$ holds for every collection of $\ell$ distinct subsets $F_{1},\dots ,F_{\ell}\in \mathcal{F}$ and $G_{1},\dots ,G_{\ell}\in \mathcal{G}$, then $\mathcal{F}$ and $\mathcal{G}$ are said to be $\ell$-weakly cross $t$-intersecting. They proved that under the $\ell$-weakly cross $t$-intersecting condition, when $n$ is sufficiently large, the conclusion of Tokushige's conjecture still holds.

\begin{Theorem}[\cite{ai2026}]\label{Weakly}
Let $k,k',\ell$ and $t$ be positive integers and let $n$ be a sufficiently large integer with respect to parameters $k,k',\ell$ and $t$. Suppose that two families $\mathcal{F}\subseteq \binom{[n]}{k}$ and $\mathcal{G}\subseteq \binom{[n]}{k'}$ are $\ell$-weakly cross $t$-intersecting. Then, provided that $n$ is sufficiently large, the following bound holds
\[
|\mathcal{F}|\cdot |\mathcal{G}|\leq \binom{n-t}{k-t}\binom{n-t}{k'-t}.
\]
\end{Theorem}

They proceed by contradiction. Firstly, they assumed that the product size is greater than the above bound. Then, by using the properties of sunflowers and through the deletion method, they proved Lemma \ref{setsunflower} and Corollary \ref{Cor1}. Subsequently, using the Erdős matching theorem, which requires $n$ to be sufficiently large, they  estimated the upper bound of the size of $\mathcal{F}$ and $\mathcal{G}$.  Finally, they derived a contradiction by utilizing the property of $\ell$-weakly cross $t$-intersection. We present a new proof of Theorem~\ref{Weakly} that avoids the Erdős matching theorem, and establish an explicit lower bound on $n$ for which the theorem is valid.

Notice that the case $\ell=1$ of $\ell$-weakly cross $t$-intersecting coincides precisely with the usual definition of cross $t$-intersecting. Since the cross $t$-intersecting case has already been proved, we may assume $\ell \geq 2$.

\begin{Theorem}\label{Set}
Let $k, k', \ell$, and $t$ be positive integers such that $k \geq k' \geq t+1$ and $\ell \geq 2$. Let $n$ be an integer satisfying $n \geq \frac{k^{2}\ell^{4}}{2}\binom{2k}{t+1}\binom{k}{t}+t.$ Suppose that two families $\mathcal{F} \subseteq \binom{[n]}{k}$ and $\mathcal{G} \subseteq \binom{[n]}{k'}$ are $\ell$-weakly cross $t$-intersecting. Then
\[
|\mathcal{F}| \cdot |\mathcal{G}| \le \binom{n-t}{k-t} \binom{n-t}{k'-t}.
\]
The equality holds only if $\mathcal{F}=\{F\in \binom{[n]}{k}\colon T\subseteq F\}$ and  $\mathcal{G}=\{G\in \binom{[n]}{k'}\colon T\subseteq G\}$ for some $t$-subset $T$ of $[n]$.
\end{Theorem}

Intersection problems are studied on some other mathematical objects, for example, vector spaces. Let $V$ be an $n$-dimensional vector space over the finite field $\mathbb{F}_{q}$ of order $q$.  Write the family of all $k$-subspaces of $V$ as $\genfrac{[}{]}{0pt}{}{V}{k}$. Recall that for any positive integers $a$ and $b$ the Gaussian binomial coefficient is defined by $\genfrac{[}{]}{0pt}{}{a}{b}_{q}= \prod_{i=0}^{b-1} \frac{q^{\,a-i\,}-1}{q^{\,b-i\,}-1}.$ In addition, we set
$\genfrac{[}{]}{0pt}{}{a}{0}_{q}=1$, and $\genfrac{[}{]}{0pt}{}{a}{b}_{q}=0$ if $b$ is a negative integer. The size of $\genfrac{[}{]}{0pt}{}{V}{k}$ is equal to $\genfrac{[}{]}{0pt}{}{n}{k}_{q}$. For brevity, we suppress $q$ from the notation in the following. 

For a positive integer $t$, a subspace family $\mathcal{F}\subseteq \genfrac{[}{]}{0pt}{}{V}{k}$ is said to be $t$-intersecting if $\dim(F\cap F')\geq t$ for any $F,F'\in \mathcal{F}$. Some extremal results of intersecting families in the set setting also have corresponding subspace versions (see \cite{FW1986,Hsieh1975,CLWZ2020,BBCFMPS}). We say that two families $\mathcal{F}\subseteq \genfrac{[}{]}{0pt}{}{V}{k}$ and $\mathcal{G}\subseteq \genfrac{[}{]}{0pt}{}{V}{k'}$ are cross $t$-intersecting if $\dim(F\cap G)\geq t$ for any $F\in \mathcal{F}$ and $G\in \mathcal{G}$. In \cite{Cao2023}, Cao et al. gave an upper bound on $\prod_{i=1}^{r} |\mathcal{F}_{i}|$ for the $r$-cross $t$-intersecting subspace family and characterized the extremal structure of the subspace family when the upper bound is achieved. We refer the readers to \cite{Cao2023,ST2014} for more details on cross $t$-intersecting subspace families.

\begin{Theorem}[\cite{Cao2023}]\label{Cross}
Let $n, r, k_{1}, k_{2}, \ldots, k_{r}$ and $t$ be positive integers with $r\geq 2$, $k_{1}\geq k_{2}\geq \cdots \geq k_{r}\geq t$
and $n\geq k_{1}+k_{2}+t+1$. If $\mathcal{F}_{1}\subseteq \genfrac{[}{]}{0pt}{}{V}{k_{1}}$, $\mathcal{F}_{2}\subseteq \genfrac{[}{]}{0pt}{}{V}{k_{2}}$, $\ldots$, $\mathcal{F}_{r}\subseteq \genfrac{[}{]}{0pt}{}{V}{k_{r}}$ are $r$-cross $t$-intersecting
families, then \[
\prod_{i=1}^{r} |\mathcal{F}_{i}| \;\leq\; \prod_{i=1}^{r} \genfrac{[}{]}{0pt}{}{n-t}{k_{i}-t}.
\]
The equality holds only if $\mathcal{F}_{i}=\{F\in \genfrac{[}{]}{0pt}{}{V}{k_{i}} \colon T\subseteq F\}$ $(i=1,2,\ldots,r)$ for some $t$-dimensional subspace $T$.
\end{Theorem}

\begin{Remark}
When $r = 2$, $k_{1}=k_{2}=k\geq t$, and $\mathcal{F}_{1}=\mathcal{F}_{2}=\mathcal{F}$, 
Theorem~\ref{Cross} yields: For $n \geq 2k+t+1$, if $\dim(F\cap F')\geq t$ for any $F,F'\in \mathcal{F}$, then the following bounds hold:
\[
|\mathcal{F}|^2 \leq \genfrac{[}{]}{0pt}{}{n-t}{k-t}^2,   |\mathcal{F}| \leq \genfrac{[}{]}{0pt}{}{n-t}{k-t}.
\]
\end{Remark}

In this paper, we investigate a weakening of the condition in Theorem \ref{Cross}. Let $n$, $k$, $k'$ be positive integers and let $\mathcal{F}\subseteq \genfrac{[}{]}{0pt}{}{V}{k}$ and $\mathcal{G}\subseteq \genfrac{[}{]}{0pt}{}{V}{k'}$. Given positive integers $\ell$ and $t$,  if 
\[
\sum_{1\leq i,j\leq \ell}\dim(F_{i}\cap G_{j})\geq \ell^2 t-\ell+1 
\]
holds for every choice of $\ell$ distinct members $F_{1},\dots ,F_{\ell}\in \mathcal{F}$ and $G_{1},\dots ,G_{\ell}\in \mathcal{G}$, then we say $\mathcal{F}$ and $\mathcal{G}$ are $\ell$-weakly cross $t$-intersecting. Notice that the case $\ell=1$ of $\ell$-weakly cross $t$-intersecting coincides precisely with the usual definition of cross $t$-intersecting.

We prove the following theorem, which is a generalization of Theorems \ref{Weakly} and \ref{Cross}. Our argument is motivated by the approach developed in \cite{ai2026}.

\begin{Theorem}\label{Main}
Let $k$, $k'$, $\ell$ and $t$ be positive integers with $k \geq k' \geq t+1$, $\ell\geq 2$, and let $n\geq (2k-t+1)(t+1)+(k-t+1)k'+k+2\ell-1$. Let $V$ be an $n$-dimensional vector space over $\mathbb{F}_{q}$. Suppose that two families $\mathcal{F}\subseteq \genfrac{[}{]}{0pt}{}{V}{k}$ and $\mathcal{G}\subseteq \genfrac{[}{]}{0pt}{}{V}{k'}$ are $\ell$-weakly cross $t$-intersecting. Then the following inequality holds
\[
|\mathcal{F}|\cdot |\mathcal{G}|\leq \genfrac{[}{]}{0pt}{}{n-t}{k-t}\genfrac{[}{]}{0pt}{}{n-t}{k'-t}.
\]
The equality holds only if $\mathcal{F}=\{F\in \genfrac{[}{]}{0pt}{}{V}{k}\colon T\subseteq F\}$ and $\mathcal{G}=\{G\in \genfrac{[}{]}{0pt}{}{V}{k'}\colon T\subseteq G\}$ for some $t$-dimensional subspace  $T$ of $V$.
\end{Theorem}

The rest of the paper is organized as follows. In Section 2, we prove several necessary inequalities. In Sections 3 and 4, we prove Theorem~\ref{Set} and Theorem~\ref{Main}, respectively. In Section 5, we conclude the paper. 
\section{Some inequalities}
For the sake of brevity in the subsequent proof, we first need to prove several inequalities.

The following lemma can be easily proved.
\begin{Lemma}\label{BinomialProperty}
	Let $m$ and $i$ be non-negative integers with $i \leq m$. Then the following hold.
	\begin{itemize}
		\item [{\rm (i)}]  ${m\choose i}={m\choose m-i}$ and ${m\choose 0}<{m\choose 1}<\cdots <{m\choose \lfloor \frac{m}{2}\rfloor}$;
		\item [{\rm (ii)}] ${m+1\choose i+1}> {m\choose i}$ if $m>i$.
	\end{itemize}
\end{Lemma}

\begin{Lemma}\label{setmono}
	Let $n$, $k$, $k'$, $t$ and $h$ be positive integers satisfying $n\geq k^{2}+2k$, $k \geq k' \geq t+1$,  and $t \leq h<k'$. Then the function
	$f(h) := \binom{k}{h}\binom{n-k}{k'-h}$ is strictly decreasing in $h$.
\end{Lemma}

\begin{proof}
	Consider the ratio of two consecutive terms

$$R:= \frac{f(h)}{f(h+1)}= \frac{\binom{k}{h} \binom{n-k}{k'-h}}{\binom{k}{h+1}\binom{n-k}{k'-h-1}}.$$
	Using the identities $\binom{k}{h+1} = \frac{k-h}{h+1} \binom{k}{h}$ and $\binom{n-k}{k'-h} =\frac{n-k-k'+h+1}{k'-h}\binom{n-k}{k'-h-1}$, we obtain
	
$$R =\frac{h+1}{k-h}\cdot \frac{n-k-k'+h+1}{k'-h}.$$

	Since $k\geq k'\geq h+1 \geq t+1 \geq 2$ and $n \geq k^{2}+2k$, the ratio $R$ satisfies $R\geq 2\cdot \frac{n-k-k'+h+1}{(k-h)(k'-h)}> 2$. Thus, $f(h) > f(h+1)$ for every $h$ in the range $t \leq h \leq k'-1$, which implies that $f(h)$ is strictly decreasing in $h$. This completes the proof. \qed
\end{proof}

\begin{Lemma}\label{setupperbound}
	Let $n, k, k', \ell, t$ and $m$ be non-negative integers satisfying $k \geq k' \geq t+1$, $\ell \geq 2$, $n\geq \frac{k^{2}\ell^{3}}{2}\binom{2k}{t+1}\binom{k}{t}+t$, and $0 \leq m \leq k'-t-1$. Then the following inequalities hold
	
	\[
	\begin{aligned}
		&2\ell-2+(\ell+k'\ell-1)\binom{\ell}{2}\binom{k'-1}{t}\left( \binom{n-t}{k-t}-\binom{n-k}{k-t} + 1 \right)\\
		&\qquad + \left(\binom{\ell}{2}\binom{2k'}{t+1}+\ell \binom{k'}{t+1}\right)\binom{n-t-1}{k-t-1}< \frac{\binom{n-t}{k'-t}\binom{n-t}{k-t}}{\binom{n-t-m}{k'-t-m}},
	\end{aligned}
	\]
	
	\[
	\begin{aligned}
		&2\ell-2+(\ell+k\ell-1)\binom{\ell}{2}\binom{k-1}{t}\left( \binom{n-t}{k'-t}-\binom{n-k'}{k'-t} + 1 \right) \\
		&\qquad + \left(\binom{\ell}{2}\binom{2k}{t+1}+\ell \binom{k}{t+1}\right)\binom{n-t-1}{k'-t-1}< \frac{\binom{n-t}{k'-t}\binom{n-t}{k-t}}{\binom{n-t-m}{k-t-m}}.
	\end{aligned}
	\]
\end{Lemma}
\begin{proof}
We only prove the first inequality, the second follows from an analogous argument.
The proof is divided into two cases.
	
	\noindent\textbf{Case 1: $k = t+1$.}  
	Under this assumption, the left-hand side of the inequality is simplified to  
     \[
  \text{LHS}=3\ell-2+2(\ell+k'\ell-1)\binom{\ell}{2}+\binom{\ell}{2} \binom{2k'}{t+1},
     \]
while the right-hand side becomes 
\[
 \text{RHS}=\frac{\binom{n-t}{k'-t}\binom{n-t}{k-t}}{\binom{n-t-m}{k'-t-m}}= \binom{n-t}{k-t}=n-t.
\]
Since $k=t+1\geq 2$ and $\ell \geq 2$, we have $3\ell-2< (k+1)\ell^{2}$ and $\frac{(k^{2}-1)}{2}\binom{k}{t}= \frac{(k^{2}-1)}{2}(t+1)\geq k+1$. According to $n\geq \frac{k^{2}\ell^{3}}{2}\binom{2k}{t+1}\binom{k}{t}+t$, it follows that
	\[
	\begin{aligned}
		\text{LHS} 
		&< (k+1)\ell^{2}+(k+1)\ell(\ell^{2}-\ell)+\binom{\ell}{2} \binom{2k}{t+1} \\
		&< (k+1)\ell^{3}+\frac{\ell^{2}}{2}\binom{2k}{t+1}\\
		&\leq \frac{(k^{2}-1)\ell^{3}}{2}\binom{2k}{t+1}\binom{k}{t}+\frac{\ell^{3}}{2}\binom{2k}{t+1}\binom{k}{t}\\
		&\leq \frac{k^{2}\ell^{3}}{2}\binom{2k}{t+1}\binom{k}{t}\leq n-t=\text{RHS}.
	\end{aligned}
	\]
	
	\noindent\textbf{Case 2: $k>t+1$.}  
	Since $n\geq \frac{k^{2}\ell^{3}}{2}\binom{2k}{t+1}\binom{k}{t}+t$, $\ell\geq 2$ and $k-t-1\geq 1$, by Lemma \ref{BinomialProperty} we have  $\binom{n-t-1}{k-t-1}\geq \binom{n-t-1}{1}>2\ell-2+(k'+1)\ell\binom{\ell}{2}\binom{k'-1}{t}$ and $\frac{\ell}{2}\binom{2k'}{t+1}\geq \ell\binom{k'}{t+1}+1$. It follows that
	\[
	\begin{aligned}
		\text{LHS} 
		&< (k'+1)\ell\binom{\ell}{2}\binom{k'-1}{t}\left(\binom{n-t}{k-t}-\binom{n-k}{k-t}\right)+ \left(\binom{\ell}{2}\binom{2k'}{t+1}+\ell \binom{k'}{t+1}+1\right)\binom{n-t-1}{k-t-1}\\
		&\leq \frac{(k'+1)\ell^{3}}{2(k-t)!}\binom{k'-1}{t}\frac{\frac{(n-t)!}{(n-k)!}-\frac{(n-k)!}{(n-2k+t)!}}{\binom{n-t-1}{k-t-1}}\binom{n-t-1}{k-t-1}+ \left(\binom{\ell}{2}\binom{2k'}{t+1}+\frac{\ell}{2}\binom{2k'}{t+1}\right)\binom{n-t-1}{k-t-1}\\
		&\leq \frac{(k'+1)\ell^{3}}{2(k-t)}\binom{k'-1}{t}\frac{\frac{(n-t)!}{(n-k)!}-\frac{(n-k)!}{(n-2k+t)!}}{\frac{(n-t-1)!}{(n-k)!}}\binom{n-t-1}{k-t-1}+ \frac{\ell^{2}}{2}\binom{2k'}{t+1}\binom{n-t-1}{k-t-1}\\
		&\leq \Bigg[\frac{(k'+1)\ell^{3}}{2(k-t)}\binom{k'-1}{t}\left(n-t-(n-k)\prod_{i=1}^{k-t-1}(1-\frac{k-t}{n-k+i})\right)+ \frac{\ell^{2}}{2}\binom{2k}{t+1}\Bigg] \binom{n-t-1}{k-t-1}\\
\end{aligned}
	\]
Since $-\frac{k-t}{n-k+1}>-1$, by Bernoulli inequality $(1+p)^{m}\geq 1+mp$ we obtain 
\[
(n-k)\prod_{i=1}^{k-t-1}(1-\frac{k-t}{n-k+i})\geq (1-\frac{k-t}{n-k+1})^{k-t-1}(n-k)\geq (1-\frac{(k-t-1)(k-t)}{n-k+1})(n-k),
\]
and
	\[
	\begin{aligned}
		\text{LHS} 
		&< \frac{(k'+1)\ell^{3}}{2(k-t)}\binom{k'-1}{t}\left((n-t)-(1-\frac{(k-t-1)(k-t)}{n-k+1})(n-k)\right)\binom{n-t-1}{k-t-1}\\
		&\qquad+\frac{\ell^{2}}{2}\binom{2k}{t+1}\binom{n-t-1}{k-t-1}\\
		&< \Bigg[\frac{(k-t)(k'+1)\ell^{3}}{2}\binom{k'-1}{t}+\frac{\ell^{2}}{2}\binom{2k}{t+1}\Bigg]\cdot \frac{k-t}{n-t}\binom{n-t}{k-t}.
\end{aligned}
	\]
Since $k\geq k'\geq t+1$, $k>t+1\geq 2$ and $\ell \geq 2$, we have $(k-t)(k'+1)\leq k^2-1$ and $k-t<k\leq \min\{\binom{k}{t},\binom{2k}{t+1}\}$ by Lemma \ref{BinomialProperty}. According to $n\geq \frac{k^{2}\ell^{3}}{2}\binom{2k}{t+1}\binom{k}{t}+t$, it follows from Lemma \ref{BinomialProperty} that	
\[
	\begin{aligned}
		\text{LHS} 
		&< \frac{\Bigg[ \frac{(k^2-1)\ell^3}{2}\binom{2k}{t+1}\binom{k}{t}+\frac{\ell^3}{2}\binom{2k}{t+1}\binom{k}{t}\Bigg]}{n-t}\binom{n-t}{k-t}\\
		&= \frac{\frac{k^{2}\ell^{3}}{2}\binom{2k}{t+1}\binom{k}{t}}{n-t}\binom{n-t}{k-t}\leq \binom{n-t}{k-t} \leq \text{RHS}. \\
	\end{aligned}
	\]
	This completes the proof of the first inequality.      \qed  
\end{proof}

\begin{Lemma}\label{setupperbound2}
	Let $n, k, k', \ell$ and $t$ be positive integers satisfying $n\geq \frac{k^{2}\ell^{4}}{2}\binom{2k}{t+1}\binom{k}{t}+t$, $k \geq k' \geq t+1$, and $\ell \geq 2$. Then the following inequalities hold
	\[
	\frac{\binom{n-t}{k-t}\binom{n-t}{k'-t}}{\left(\binom{\ell}{2}\binom{k'-1}{t}+2\right)(\ell-1)+\left(\binom{\ell}{2}\binom{2k'}{t+1}+\ell\binom{k'}{t+1}\right)\binom{n-t-1}{k-t-1}}\geq \ell\sum_{h=t}^{k'} \binom{k}{h}\binom{n-k}{k'-h}+\ell,
	\]
	\[
	\frac{\binom{n-t}{k-t}\binom{n-t}{k'-t}}{\left(\binom{\ell}{2}\binom{k-1}{t}+2\right)(\ell-1)+\left(\binom{\ell}{2}\binom{2k}{t+1}+\ell\binom{k}{t+1}\right)\binom{n-t-1}{k'-t-1}}\geq \ell\sum_{h=t}^{k'}\binom{k'}{h}\binom{n-k'}{k-h}+\ell.
	\]
\end{Lemma}

\begin{proof}
We only prove the first inequality, the second follows from an analogous argument.
The proof is divided into two cases.

	\noindent\textbf{Case 1: $k = t+1$.}  
	Under this assumption, the left-hand side of the inequality can be simplified to
	\[
	\text{LHS} = \frac{\binom{n-t}{1}\binom{n-t}{1}}{(\binom{\ell}{2}+2)(\ell-1)+\ell+\binom{\ell}{2}\binom{2t+2}{t+1}},
	\]
	while the right-hand side becomes
	\[
	\text{RHS} = \ell(t+1)(n-t-1) + 2\ell.
	\]
Since $t\geq 1$ and $\ell\geq 2$, we have $(\ell-1)\binom{\ell}{2}\binom{2t+2}{t+1}\geq 6\cdot (\ell-1)\binom{\ell}{2}>(\binom{\ell}{2}+2)(\ell-1)+\ell$. According to $n\geq \frac{k^{2}\ell^{4}}{2}\binom{2k}{t+1}\binom{k}{t}+t$, it follows that 
	\[
	\begin{aligned}
		\text{LHS} 
		&=\frac{(n-t)^{2}}{(\binom{\ell}{2}+2)(\ell-1)+\ell+\binom{\ell}{2}\binom{2t+2}{t+1}}>  \frac{(n-t)^{2}}{\ell\binom{\ell}{2}\binom{2t+2}{t+1}}\\
		&\geq \frac{(n-t)\frac{k^{2}\ell^{4}}{2}\binom{2t+2}{t+1}\binom{t+1}{t}}{\ell\frac{\ell(\ell-1)}{2}\binom{2t+2}{t+1}}\geq \frac{(n-t)k^{2}\ell^{4}\binom{2t+2}{t+1}\binom{t+1}{t}}{\ell^{3}\binom{2t+2}{t+1}}\\
		&=  k^{2}\ell(t+1)(n-t)> 2\ell(t+1)(n-t)\\
		&\geq \text{RHS}.
	\end{aligned}
	\]
	
	\noindent\textbf{Case 2: $k > t+1$.}  Since $n\geq \frac{k^{2}\ell^{3}}{2}\binom{2k}{t+1}\binom{k}{t}+t$, $\ell\geq 2$ and $k-t-1\geq 1$, by Lemma \ref{BinomialProperty} we have  $\binom{n-t-1}{k-t-1}\geq \binom{n-t-1}{1}=n-t-1>\left(\binom{\ell}{2}\binom{k'-1}{t}+2\right)(\ell-1)$ and $\frac{\ell}{2}\binom{2k'}{t+1}\geq \ell\binom{k'}{t+1}+1$. It follows that
	\[
	\begin{aligned}
		\text{LHS} 
		&\geq  \frac{\binom{n-t}{k-t}\binom{n-t}{k'-t}}{\left(\binom{\ell}{2}\binom{2k'}{t+1}+\ell\binom{k'}{t+1}+1\right)\binom{n-t-1}{k-t-1}}\geq    \frac{\binom{n-t}{k'-t}(n-t)}{\left(\binom{\ell}{2}\binom{2k'}{t+1}+\frac{\ell}{2}\binom{2k'}{t+1}\right)(k-t)} \\
		&\geq \frac{n-t}{\frac{\ell^{2}}{2}\binom{2k'}{t+1}(k-t)} \binom{n-t}{k'-t}\geq \frac{\frac{\ell^{3}}{2}\binom{2k'}{t+1}k^{2}\binom{k}{t}}{\frac{\ell^{2}}{2}\binom{2k'}{t+1}(k-t)} \binom{n-t}{k'-t}\\
		&= \frac{k^{2}\ell}{k-t}\binom{k}{t}\binom{n-t}{k'-t}\geq  (k'-t+1)\ell\binom{k}{t}\binom{n-k}{k'-t}+\ell.
	\end{aligned}
	\]
	By Lemma \ref{setmono}, we obtain $(k'-t+1)\ell\binom{k}{t}\binom{n-k}{k'-t}+\ell\geq \text{RHS}$, which completes the proof of the first inequality.
\qed
\end{proof}

\begin{Lemma}[\cite{GM2016}]\label{Spacenumber} Let $W$ be a subspace of dimension $k$ in a vector space of dimension $n$ over $\mathbb{F}_{q}$. The number of $m$-dimensional subspaces whose intersection with $W$ has dimension $h$ is 
	\[
	q^{(k-h)(m-h)} \genfrac{[}{]}{0pt}{}{k}{h}\genfrac{[}{]}{0pt}{}{n-k}{m-h}.
	\]
\end{Lemma}

The following lemma can be easily proved.

\begin{Lemma}\label{GaussianBinomialProperty}
Let $m$ and $i$ be positive integers with $i \leq m$. Then the following hold.
\begin{itemize}
	\item [{\rm (i)}]  $\genfrac{[}{]}{0pt}{}{m}{i}=\frac{q^m-1}{q^i-1}\genfrac{[}{]}{0pt}{}{m-1}{i-1}$; $\genfrac{[}{]}{0pt}{}{m}{0}<\genfrac{[}{]}{0pt}{}{m}{1}<\cdots <\genfrac{[}{]}{0pt}{}{m}{\lfloor\frac{m}{2}\rfloor}$;
	\item [{\rm (ii)}] $q^{m-i}<\frac{q^m-1}{q^i-1}<q^{m-i+1}$ and $q^{i-m-1}<\frac{q^i-1}{q^m-1}<q^{i-m}$ if $i<m$;
	\item [{\rm (iii)}] $q^{i(m-i)}\leq \genfrac{[}{]}{0pt}{}{m}{i}<q^{i(m-i+1)}$, and $q^{i(m-i)}< \genfrac{[}{]}{0pt}{}{m}{i}$ if $i<m$.
\end{itemize}

\end{Lemma} 

\begin{Lemma}\label{mono}
	Let $n$, $k$, $k'$, $t$ and $h$ be positive integers satisfying $n\geq k+k'-t$, $k \geq k' \geq t+1$,  and $t \leq h < k'$. Then the function
	\[
	F(h) := q^{(k-h)(k'-h)} \genfrac{[}{]}{0pt}{}{k}{h}\genfrac{[}{]}{0pt}{}{n-k}{k'-h}
	\]
	is strictly decreasing in $h$.
\end{Lemma}

\begin{proof}
	Consider the ratio of two consecutive terms:
	\[
	R(h):= \frac{F(h)}{F(h+1)}
	= \frac{ q^{(k-h)(k'-h)} \genfrac{[}{]}{0pt}{}{k}{h} \genfrac{[}{]}{0pt}{}{n-k}{k'-h}}
	{ q^{(k-h-1)(k'-h-1)} \genfrac{[}{]}{0pt}{}{k}{h+1}\genfrac{[}{]}{0pt}{}{n-k}{k'-h-1}}.
	\]
	Using the identities $\genfrac{[}{]}{0pt}{}{k}{h+1} = \frac{q^{k-h}-1}{q^{h+1}-1} \genfrac{[}{]}{0pt}{}{k}{h}$ and $\genfrac{[}{]}{0pt}{}{n-k}{k'-h} = \frac{q^{n-k-k'+h+1}-1}{q^{k'-h}-1}\genfrac{[}{]}{0pt}{}{n-k}{k'-h-1}$ , we obtain

	$$R(h) = q^{k+k'-2h-1}\cdot \frac{q^{h+1}-1}{q^{k-h}-1}\cdot \frac{q^{n-k-k'+h+1}-1}{q^{k'-h}-1}\geq \frac{1}{q} \cdot (q^{h+1}-1) \cdot (q^{n-k-k'+h+1}-1).$$

	Since $h+1 \geq t+1 \geq 2$ and $n \geq k+k'-t \geq k+k'-h$, the factors satisfy $\frac{q^{h+1}-1}{q} \geq \frac{q^2-1}{q} > 1$ and $q^{n-k-k'+h+1}-1 \geq  q-1 \geq 1$. Thus, $F(h) > F(h+1)$ for every $h$ in the range $t \leq h \leq k'-1$, which implies that $F(h)$ is strictly decreasing in $h$. This completes the proof. \qed
\end{proof}

\begin{Lemma}\label{Upperbound}
	Let $n$, $k$, $k'$, $\ell$, $t$ and $m$ be non-negative integers satisfying $n\geq (2k-t)(t+1)+k+\ell+2$, $k \geq k' \geq t+1\geq 2$, $\ell \geq 2$, and $0 \leq m \leq k'-t-1$. Then the following inequalities hold
	\[
	\begin{aligned}
		&2\ell-2+ \left(\genfrac{[}{]}{0pt}{}{k'}{1}\ell+\ell-1\right)\binom{\ell}{2}\genfrac{[}{]}{0pt}{}{k'-1}{t}\left(\genfrac{[}{]}{0pt}{}{n-t}{k-t} - q^{(k-t)^{2}}\genfrac{[}{]}{0pt}{}{n-k}{k-t}+1\right) \\
		&\qquad + \left(\binom{\ell}{2}\genfrac{[}{]}{0pt}{}{2k'}{t+1}+\ell \genfrac{[}{]}{0pt}{}{k'}{t+1}\right)\genfrac{[}{]}{0pt}{}{n-t-1}{k-t-1}< \frac{\genfrac{[}{]}{0pt}{}{n-t}{k'-t}\genfrac{[}{]}{0pt}{}{n-t}{k-t}}{\genfrac{[}{]}{0pt}{}{n-t-m}{k'-t-m}},
	\end{aligned}
	\]
	\[
	\begin{aligned}
		&2\ell-2+ \left(\genfrac{[}{]}{0pt}{}{k}{1}\ell+\ell-1\right)\binom{\ell}{2}\genfrac{[}{]}{0pt}{}{k-1}{t}\left(\genfrac{[}{]}{0pt}{}{n-t}{k'-t} - q^{(k'-t)^{2}}\genfrac{[}{]}{0pt}{}{n-k'}{k'-t}+1\right) \\
		&\qquad + \left(\binom{\ell}{2}\genfrac{[}{]}{0pt}{}{2k}{t+1}+\ell \genfrac{[}{]}{0pt}{}{k}{t+1}\right)\genfrac{[}{]}{0pt}{}{n-t-1}{k'-t-1}< \frac{\genfrac{[}{]}{0pt}{}{n-t}{k'-t}\genfrac{[}{]}{0pt}{}{n-t}{k-t}}{\genfrac{[}{]}{0pt}{}{n-t-m}{k-t-m}}.
	\end{aligned}
	\]
\end{Lemma}
\begin{proof}
We only prove the first inequality, the second follows from an analogous argument.
The proof is divided into two cases.
	
	\noindent\textbf{Case 1: $k= t+1$.}  
	Under this assumption, the left-hand side of the inequality can be simplified to 
\[
\begin{aligned}
\text{LHS}&=2\ell-2+(\frac{q^{t+1}-1}{q-1}\ell+\ell-1)\binom{\ell}{2}(\frac{q^{n-t}-1}{q-1}-\frac{q^{n-t}-q}{q-1}+1) + \binom{\ell}{2} \genfrac{[}{]}{0pt}{}{2k'}{t+1}+\ell\\
&=3\ell-2+2(\frac{q^{t+1}-1}{q-1}\ell+\ell-1)\binom{\ell}{2}+ \binom{\ell}{2} \genfrac{[}{]}{0pt}{}{2k'}{t+1},
\end{aligned}
\]
while the right-hand side becomes 
\[
\text{RHS}=\frac{\genfrac{[}{]}{0pt}{}{n-t}{k'-t}\genfrac{[}{]}{0pt}{}{n-t}{k-t}}{\genfrac{[}{]}{0pt}{}{n-t-m}{k'-t-m}}= \genfrac{[}{]}{0pt}{}{n-t}{k-t}=\frac{q^{n-t}-1}{q-1}.
\]
Since $t\geq 1$ and $\ell\geq 2$, by Lemma \ref{GaussianBinomialProperty} we have $3\ell-2+(q^{t+1}\ell-1)(\ell^{2}-\ell)<3\ell-2+q^{t+1}\ell(\ell^{2}-\ell)<q^{t+1}\ell^{3}$,  $\ell^{3}\leq q^{\ell+2}$ and $\frac{\ell^{2}}{2}\leq q^{\ell}$. According to $n\geq (2k-t)(t+1)+k+\ell+2$, it follows from Lemma \ref{GaussianBinomialProperty} that
	\[
	\begin{aligned}
		\text{LHS} 
		&< 3\ell-2+(q^{t+1}\ell-1)(\ell^{2}-\ell)+ \binom{\ell}{2} \genfrac{[}{]}{0pt}{}{2k}{t+1}\\
		&\leq q^{t+1}\ell^{3}+ \frac{\ell^{2}}{2}q^{(2k-t)(t+1)}\leq q^{t+\ell+3} 
		+ q^{(2k-t)(t+1)+\ell} \\
		&\leq q^{(2k-t)(t+1) + \ell + 1} 
		< q^{n-t-1}\leq \frac{q^{n-t} - 1}{q-1} = \text{RHS}.
	\end{aligned}
	\]
	
	\noindent\textbf{Case 2: $k > t+1$.}  
	Since $n\geq (2k-t)(t+1)+k+\ell+2$, $\ell\geq 2$ and $k-t-1\geq 1$, by Lemma \ref{GaussianBinomialProperty} we have $\frac{\ell}{2}\genfrac{[}{]}{0pt}{}{2k'}{t+1}\geq \ell\genfrac{[}{]}{0pt}{}{k'}{t+1}+1$, $q^k\geq  \genfrac{[}{]}{0pt}{}{k}{1}+1$ and 
\[
\begin{aligned}
\genfrac{[}{]}{0pt}{}{n-t-1}{k-t-1}&\geq \frac{q^{n-t-1}-1}{q-1}>q^{n-t-2}\\
&\geq  q^{(2k-t-1)(t+1)+k+\ell+1}\\
&\geq q^{(k+1)(t+1)+k+\ell+1}\geq \frac{\ell^{3}}{2}q^{k}\cdot q^{k't}\\
&\geq \frac{\ell^{2}}{2}q^{k}\genfrac{[}{]}{0pt}{}{k'-1}{t}+q^{k}\ell\binom{\ell}{2}\genfrac{[}{]}{0pt}{}{k'-1}{t}\\
&\geq 2\ell-2+ (\genfrac{[}{]}{0pt}{}{k'}{1}\ell+\ell-1)\binom{\ell}{2}\genfrac{[}{]}{0pt}{}{k'-1}{t}.
\end{aligned}
\]
It follows from Lemma \ref{GaussianBinomialProperty} that
	\[
	\begin{aligned}
		\text{LHS} 
		&<  \left (\genfrac{[}{]}{0pt}{}{k'}{1}\ell+\ell-1\right )\binom{\ell}{2}\genfrac{[}{]}{0pt}{}{k'-1}{t}\left(\genfrac{[}{]}{0pt}{}{n-t}{k-t} - q^{(k-t)^{2}}\genfrac{[}{]}{0pt}{}{n-k}{k-t}\right) \\
		&\qquad + \left(\binom{\ell}{2}\genfrac{[}{]}{0pt}{}{2k'}{t+1}+\ell \genfrac{[}{]}{0pt}{}{k'}{t+1}+1\right)\genfrac{[}{]}{0pt}{}{n-t-1}{k-t-1}\\
		&\leq (\frac{q^{k'}-1}{q-1}\ell+\ell-1)\binom{\ell}{2}\genfrac{[}{]}{0pt}{}{k'-1}{t}\left(\genfrac{[}{]}{0pt}{}{n-t}{k-t} - \genfrac{[}{]}{0pt}{}{n-t}{k-t}\prod_{i=0}^{k-t-1}(1-\frac{q^{k-t}-1}{q^{n-t-i}-1})\right) \\
		&\qquad + \left(\binom{\ell}{2}\genfrac{[}{]}{0pt}{}{2k'}{t+1}+\frac{\ell}{2}\genfrac{[}{]}{0pt}{}{2k'}{t+1}\right)\genfrac{[}{]}{0pt}{}{n-t-1}{k-t-1} \\
		&\leq q^{k'}\ell\binom{\ell}{2}\genfrac{[}{]}{0pt}{}{k'-1}{t}\left(\genfrac{[}{]}{0pt}{}{n-t}{k-t} - \genfrac{[}{]}{0pt}{}{n-t}{k-t}\prod_{i=0}^{k-t-1}(1-\frac{q^{k-t}-1}{q^{n-t-i}-1})\right)+\frac{\ell^{2}}{2}\genfrac{[}{]}{0pt}{}{2k'}{t+1}\genfrac{[}{]}{0pt}{}{n-t-1}{k-t-1} \\
		&\leq \frac{\ell^{3}}{2}q^{k'}q^{(k'-t)t}\left(1- \prod_{i=0}^{k-t-1}(1-\frac{q^{k-t}-1}{q^{n-t-i}-1})\right)\genfrac{[}{]}{0pt}{}{n-t}{k-t}+ \frac{\ell^{2}}{2}\genfrac{[}{]}{0pt}{}{2k}{t+1}\frac{\genfrac{[}{]}{0pt}{}{n-t}{k-t}}{q^{n-k}}\\
		&\leq q^{k'+\ell+1+(k'-t)t}\left(1- \prod_{i=0}^{k-t-1}(1-\frac{q^{k-t}-1}{q^{n-k+1}-1})\right)\genfrac{[}{]}{0pt}{}{n-t}{k-t}+ \frac{\genfrac{[}{]}{0pt}{}{n-t}{k-t}q^{(2k-t)(t+1)+\ell}}{q^{(2k-t)(t+1)+\ell+2}} \\
		&\leq  q^{k'+\ell+1+(k'-t)t}\left(1-(1-q^{2k-t-1-n})^{k-t}\right)\genfrac{[}{]}{0pt}{}{n-t}{k-t}+ \frac{\genfrac{[}{]}{0pt}{}{n-t}{k-t}}{q^{2}}.
\end{aligned}
	\]
Since  $-q^{2k-t-1-n}>-1$, by Bernoulli inequality $(1+p)^{m}\geq 1+mp$  we obtain
\[
(1-q^{2k-t-1-n})^{k-t}\geq 1-(k-t)q^{2k-t-1-n}\geq 1-(k-t)q^{-(2k-t)t-k-\ell-3},
\]
and
\[
	\begin{aligned}
         \text{LHS} 
		&<  (k-t)q^{k'+\ell+1+(k'-t)t}q^{-(2k-t)t-k-\ell-3}\genfrac{[}{]}{0pt}{}{n-t}{k-t}+ \frac{\genfrac{[}{]}{0pt}{}{n-t}{k-t}}{q^{2}}\\
		&\leq [(k-t)q^{-kt-2}+\frac{1}{q^{2}}]\genfrac{[}{]}{0pt}{}{n-t}{k-t}< \genfrac{[}{]}{0pt}{}{n-t}{k-t}\leq\text{RHS},
	\end{aligned}
	\]
	which completes the  proof of the first inequality.   
 \qed
\end{proof}

\begin{Lemma}\label{Upperbound2}
	Let $n, k, k', \ell$ and $t$ be positive integers satisfying $n\geq (2k-t+1)(t+1)+(k-t+1)k'+k+2\ell-1$, $k \geq k' \geq t+1$, and $\ell \geq 2$. Then the following inequalities hold
	\[
	\frac{\genfrac{[}{]}{0pt}{}{n-t}{k-t}\genfrac{[}{]}{0pt}{}{n-t}{k'-t}}{\left(\binom{\ell}{2}\genfrac{[}{]}{0pt}{}{k'-1}{t}+2\right)(\ell-1)+\left(\binom{\ell}{2}\genfrac{[}{]}{0pt}{}{2k'}{t+1}+\ell\genfrac{[}{]}{0pt}{}{k'}{t+1}\right)\genfrac{[}{]}{0pt}{}{n-t-1}{k-t-1}}\geq \ell\sum_{h=t}^{k'} q^{(k-h)(k'-h)} \genfrac{[}{]}{0pt}{}{k}{h}\genfrac{[}{]}{0pt}{}{n-k}{k'-h}+\ell,
	\]
	\[
	\frac{\genfrac{[}{]}{0pt}{}{n-t}{k-t}\genfrac{[}{]}{0pt}{}{n-t}{k'-t}}{\left(\binom{\ell}{2}\genfrac{[}{]}{0pt}{}{k-1}{t}+2\right)(\ell-1)+\left(\binom{\ell}{2}\genfrac{[}{]}{0pt}{}{2k}{t+1}+\ell\genfrac{[}{]}{0pt}{}{k}{t+1}\right)\genfrac{[}{]}{0pt}{}{n-t-1}{k'-t-1}}\geq \ell\sum_{h=t}^{k'} q^{(k-h)(k'-h)} \genfrac{[}{]}{0pt}{}{k'}{h}\genfrac{[}{]}{0pt}{}{n-k'}{k-h}+\ell.
	\]
\end{Lemma}

\begin{proof}
We only prove the first inequality, the second follows from an analogous argument.
The proof is divided into two cases.
	
	\noindent\textbf{Case 1: $k = t+1$.}  
	Under this assumption, the left-hand side of the inequality can be simplified to
	\[
	\text{LHS} = \frac{\genfrac{[}{]}{0pt}{}{n-t}{1}\genfrac{[}{]}{0pt}{}{n-t}{1}}{\left(\binom{\ell}{2}+2\right)(\ell-1)+\ell+\binom{\ell}{2}\genfrac{[}{]}{0pt}{}{2t+2}{t+1}},
	\]
	while the right-hand side becomes
	\[
	\text{RHS}=q\ell\genfrac{[}{]}{0pt}{}{t+1}{1}\genfrac{[}{]}{0pt}{}{n-t-1}{1}+2\ell.
	\]
Since $t\geq 1$ and $\ell\geq 2$, by Lemma \ref{GaussianBinomialProperty} we have $(\ell-1)\binom{\ell}{2}\genfrac{[}{]}{0pt}{}{2t+2}{t+1}\geq (\ell-1)\binom{\ell}{2}\genfrac{[}{]}{0pt}{}{4}{2}\geq 35\cdot (\ell-1)\binom{\ell}{2}>\left(\binom{\ell}{2}+2\right)(\ell-1)+\ell$. According to $n\geq (2k-t+1)(t+1)+(k-t+1)k'+k+2\ell-1>(2k-t+1)(t+1)+k+2\ell+2$, it follows from Lemma \ref{GaussianBinomialProperty} that 
		\[
	\begin{aligned}
		\text{LHS} 
		&=\frac{\genfrac{[}{]}{0pt}{}{n-t}{1}\genfrac{[}{]}{0pt}{}{n-t}{1}}{\left(\binom{\ell}{2}+2\right)(\ell-1)+\ell+\binom{\ell}{2}\genfrac{[}{]}{0pt}{}{2t+2}{t+1}} \geq \frac{q^{2(n-t-1)}}{\ell\binom{\ell}{2}\genfrac{[}{]}{0pt}{}{2t+2}{t+1}}\\
		&\geq  \frac{q^{(2k-t)(t+1)+k+2\ell+2}\cdot q^{n-t-1}}{\frac{\ell^{3}}{2}\genfrac{[}{]}{0pt}{}{2t+2}{t+1}}\geq  \frac{q^{(2k-t)(t+1)+k+2\ell+2}\cdot q^{n-t-1}}{q^{\ell+1+(t+2)(t+1)}}\\
		&\geq q^{k+\ell+1}\cdot q^{n-t-1}=  (q^{\ell}\cdot q^{t+2}-1)q^{n-t-1}+q^{n-t-1}\\
		&\geq \ell\frac{q^{t+2}-q}{q-1}\cdot\frac{q^{n-t-1}-1}{q-1}+2\ell= \text{RHS}.
	\end{aligned}
	\]
	
	\noindent\textbf{Case 2: $k > t+1$.}  Since $n\geq (2k-t+1)(t+1)+(k-t+1)k'+k+2\ell-1$, $\ell\geq 2$ and $k-t-1\geq 1$, by Lemma \ref{GaussianBinomialProperty} we have  $\genfrac{[}{]}{0pt}{}{n-t-1}{k-t-1}\geq \genfrac{[}{]}{0pt}{}{n-t-1}{1}>q^{n-t-2}>\left(\binom{\ell}{2}\genfrac{[}{]}{0pt}{}{k'-1}{t}+2\right)(\ell-1)$, $\frac{\ell}{2}\genfrac{[}{]}{0pt}{}{2k'}{t+1}\geq \ell\genfrac{[}{]}{0pt}{}{k'}{t+1}+1$, $\frac{\ell^2}{2}\leq q^{\ell}$, $\genfrac{[}{]}{0pt}{}{2k'}{t+1}\leq q^{(2k'-t)(t+1)}$ and $q^{k'+\ell}\geq 2\ell k'\geq 2\ell(k'-t+1)$. It follows that
	\[
	\begin{aligned}
		\text{LHS} 
		&\geq  \frac{\genfrac{[}{]}{0pt}{}{n-t}{k-t}\genfrac{[}{]}{0pt}{}{n-t}{k'-t}}{\left(\binom{\ell}{2}\genfrac{[}{]}{0pt}{}{2k'}{t+1}+\ell\genfrac{[}{]}{0pt}{}{k'}{t+1}+1\right)\genfrac{[}{]}{0pt}{}{n-t-1}{k-t-1}}\geq   \frac{\genfrac{[}{]}{0pt}{}{n-t}{k'-t}(q^{n-t}-1)}{\left(\binom{\ell}{2}\genfrac{[}{]}{0pt}{}{2k'}{t+1}+\frac{\ell}{2}\genfrac{[}{]}{0pt}{}{2k'}{t+1}\right)(q^{k-t}-1)} \\
		&\geq \frac{q^{n-k}}{\frac{\ell^{2}}{2}\genfrac{[}{]}{0pt}{}{2k'}{t+1}} \genfrac{[}{]}{0pt}{}{n-t}{k'-t}\geq  \frac{q^{(2k-t+1)(t+1)+(k-t+1)k'+2\ell-1}}{q^{(2k'-t)(t+1)+\ell}}\genfrac{[}{]}{0pt}{}{n-k}{k'-t}\\
		&\geq  q^{(k-t+1)k'+\ell+t}\genfrac{[}{]}{0pt}{}{n-k}{k'-t}= q^{k'+\ell}\cdot q^{(k-t)(k'-t)}\cdot q^{(k-t+1)t}\cdot \genfrac{[}{]}{0pt}{}{n-k}{k'-t}\\
		&\geq  \ell(k'-t+1)q^{(k-t)(k'-t)}\cdot q^{(k-t+1)t}\genfrac{[}{]}{0pt}{}{n-k}{k'-t}+\ell q^{(k-t)(k'-t)}\cdot q^{(k-t+1)t}\genfrac{[}{]}{0pt}{}{n-k}{k'-t}\\
		&\geq  \ell(k'-t+1)q^{(k-t)(k'-t)} \genfrac{[}{]}{0pt}{}{k}{t}\genfrac{[}{]}{0pt}{}{n-k}{k'-t}+\ell.
	\end{aligned}
	\]

	By Lemma \ref{mono} and $n\geq k+k'-t$, we obtain $\ell(k'-t+1)q^{(k-t)(k'-t)} \genfrac{[}{]}{0pt}{}{k}{t}\genfrac{[}{]}{0pt}{}{n-k}{k'-t}+\ell\geq \text{RHS}$, which completes the proof of the first inequality.
\qed
\end{proof}

\section{Proof of Theorem \ref{Set}}

As stated in the introduction, our idea of proving  Theorem~\ref{Set} was inspired by the method in \cite{ai2026}, namely the sunflower method.
A $k$-uniform sunflower is a family of $k$-sets in which all pairs of distinct sets have the same intersection (called the kernel, the remaining parts are called petals). For positive integers $k, t, u$ with $t\leq k\leq n$, we denote by $\mathcal{S}(k,t,u)\subseteq \binom{n}{k}$ a $k$-uniform sunflower with a kernel of size $t$ and $u$ petals. 

For the proof of Theorem~\ref{Set}, we first need the following lemmas.

\begin{Lemma}[\cite{ai2026}]\label{setsunflower}
	Let $r=(1+k')\ell$. Let $\mathcal {F}\subseteq \binom{[n]}{k}$ and $\mathcal{G}\subseteq \binom{[n]}{k'}$ be $\ell$-weakly cross $t$-intersecting. Suppose that $\mathcal{F}$ contains a sunflower $\mathcal{S}(k,t,r)$ with the kernel $K$ as a subfamily. Then every member of $\mathcal{G}$ contains $K$, thereby $|\mathcal{G}|\leq \binom{n-t}{k'-t}.$
\end{Lemma}

Since the roles of $\mathcal{F}$ and $\mathcal{F}'$ are symmetric, Ai et al. obtained the following corollary as an immediate consequence of Lemma \ref{setsunflower}.

\begin{Corollary}[\cite{ai2026}]\label{Cor1}
	Let $r = (1+k')\ell$ and $r'=(1+k)\ell$. Let $\mathcal {F}\subseteq \binom{[n]}{k}$ and $\mathcal{G}\subseteq \binom{[n]}{k'}$ be $\ell$-weakly cross $t$-intersecting. Suppose that $\mathcal{F}$ contains a sunflower $\mathcal{S}(k,t,r)$ as a subfamily and that $\mathcal G$ contains a sunflower $S(k',t,r')$ as a subfamily respectively. Then $|\mathcal{F}|\cdot |\mathcal{G}| \leq \binom{n-t}{k-t}\binom{n-t}{k'-t}$.
\end{Corollary}

The conclusion of Corollary \ref{Cor1} attains the upper bound of Theorem~\ref{Set}. Therefore, we can assume that either $\mathcal{F}$ does not contain $\mathcal{S}(k,t,r)$ or $G$ does not contain $\mathcal{S}(k',t,r')$.

\begin{Lemma}\label{Nsunflowerset}
	Let $\mathcal {F}\subseteq \binom{[n]}{k}$ and $\mathcal{G}\subseteq \binom{[n]}{k'}$ be $\ell$-weakly cross $t$-intersecting. Suppose that $\mathcal{F}$ does not contain $\mathcal{S}(k,t,r)$ as a subfamily, where $r=(1+k')\ell$. Then $|\mathcal{F}|\cdot |\mathcal{G}|<\binom{n-t}{k-t}\binom{n-t}{k'-t}$.
\end{Lemma}

\begin{proof} 
We proceed by contradiction. Assume  $|\mathcal{F}|\cdot |\mathcal{G}|\geq \binom{n-t}{k-t}\binom{n-t}{k'-t},$
and we split the proof into two cases. 
	
	\paragraph{Case $1$:} There exist $\ell$ members $G_{1},G_{2},\ldots,G_{\ell}$ in $\mathcal{G}$ such that $|G_{i}\cap G_{j}|\leq t-1$ for any $i\neq j\in [\ell]$. Then we choose such $\ell$ subsets.
	
	By letting $\mathcal{E} = \{F\in \mathcal{F}\colon |G_{i}\cap F|\leq t-1 \text{ for every } i\in [\ell]\}$, we claim that $|\mathcal{E}|\leq \ell-1$.
	Otherwise, if there exist $\ell$ members $F_{1},F_{2},\ldots,F_{\ell}\in \mathcal{F}$ such that $|G_{i}\cap F_{j}|\leq t-1$ for every $i,j\in [\ell]$, then
	$$\sum_{1\leq i,j\leq \ell}|G_{i}\cap F_{j}|\leq \ell^{2}(t-1)<\ell^{2}t-\ell+1,$$
	which contradicts the assumption that $\mathcal{F}$ and $\mathcal{G}$ are $\ell$-weakly cross $t$-intersecting. 
	
	By the definition of $\mathcal{E}$, for every element $F$ in $\mathcal{F}\setminus \mathcal{E}$, there exists at least one $i \in [\ell]$ such that $|G_{i}\cap F|\geq t$. 
	Let $$\mathcal{E}_{1}=\{F\in \mathcal{F}\colon |G_{i}\cap F|=t \text{ for some } i\in [\ell], \text{ and } |G_{j}\cap F|\leq t -1\text{ for each } j\in [\ell]\backslash \{i\} \},$$
	$$\mathcal{E}_{2} = \bigl\{ F \in \mathcal{F} \colon |G_{i} \cap F| \leq t \text{ for all } i \in [\ell], \text{ and } \bigl| \{ i \in [\ell] : |G_{i} \cap F| = t \} \bigr| \ge 2 \bigr\},$$
	$$\mathcal{E}_{3}=\{F\in \mathcal{F}\colon |G_{i}\cap F|\geq t+1 \text{ for some } i\in [\ell]\}.$$
	
	Claim: $|\mathcal{E}_{1}|\leq \ell-1$.
	
	Proof of claim: We proceed by contradiction. Assume $|\mathcal{E}_{1}|\geq \ell$ and choose $\ell$ distinct elements $F_{1},F_{2},\ldots, F_{\ell}\in\mathcal{E}_{1}$. Then $$\sum_{1\leq i,j\leq \ell}|G_{i}\cap F_{j}|\leq \ell(\ell-1)(t-1)+\ell t<\ell^{2}t-\ell+1,$$
which contradicts the assumption that $\mathcal{F}$ and $\mathcal{G}$ are $\ell$-weakly cross $t$-intersecting. 
	
	Next, we estimate the sizes of $\mathcal{E}_{2} $ and $\mathcal{E}_{3}$.
	
	The members in $\mathcal{E}_{2}$ can be classified into two types. The first class consists of all $k$-sets $F$ whose intersections with those $G_i$ among $G_1, G_2, \ldots, G_\ell$ that satisfy $|F\cap G_i| = t$ are the same $t$-set, say $T$. The second class consists of all $k$-sets whose intersections with two subsets $G_{i}$ and $G_{j}$ among $G_1, G_2, \ldots, G_{\ell}$ are two distinct $t$-sets. For the first class, the number of such subsets in $\mathcal{E}_2$ is $0$ (since $F\cap G_i = F\cap G_j = T$ implies $T\subseteq G_i\cap G_j$, contradicting the choice of $G_{1},G_{2},\ldots,G_{\ell}$). 
	For the second class, it is easy to see that if a $k$-set $F$ has intersections with $G_{i}$ and $G_{j}$ that are distinct $t$-sets, then the size of its intersection with $G_{i}\cup G_{j}$ is at least $t+1$. Note that $|G_{i}\cup G_{j}|\leq 2k'$ since $G_{i}$ and $G_{j}$ are $k'$-sets. Therefore, we conclude that $$\mathcal{E}_{2}\subseteq \bigcup_{1\leq i<j\leq \ell}(\bigcup_{A\subseteq G_{i}\cup G_{j},|A|=t+1}\mathcal{F}(A)),$$
where $\mathcal{F}(A)=\{F\in \binom{[n]}{k}\colon A\subset F\}$. Then
$$|\mathcal{E}_{2}|\leq \binom{\ell}{2}\binom{2k'}{t+1} \binom{n-t-1}{k-t-1}.$$

	It is easy to see that $$\mathcal{E}_{3}\subseteq \bigcup_{i=1}^{\ell}(\bigcup_{A\subseteq G_{i},|A|=t+1}\mathcal{F}(A)).$$
    Therefore, $$|\mathcal{E}_{3}|\leq \ell \binom{k'}{t+1}\binom{n-t-1}{k-t-1}.$$
	
	Since $\mathcal{F}\subseteq \mathcal{E}\cup \mathcal{E}_{1}\cup \mathcal{E}_{2}\cup \mathcal{E}_{3}$, we have 
	$$|\mathcal{F}|\leq 2\ell-2+\left[\binom{\ell}{2}\binom{2k'}{t+1}+\ell \binom{k'}{t+1}\right]\binom{n-t-1}{k-t-1}.$$
	By  assumption $|\mathcal{F}|\cdot |\mathcal{G}| \geq \binom{n-t}{k-t}\binom{n-t}{k'-t}$ and Lemma~\ref{setupperbound2},
	we have  $$|\mathcal{G}| \geq \ell\sum_{h=t}^{k'}\binom{k}{h}\binom{n-k}{k'-h}+\ell.$$
	
	To reach a contradiction, pick $\ell$ distinct members $F_{1},F_{2},\ldots,F_{\ell}$ of $\mathcal{F}$.  
	For each fixed $j \in [\ell]$, the number of $k'$-sets intersecting $F_j$ in at least $t$ elements is at most  
	$\sum_{h=t}^{k'}\binom{k}{h}\binom{n-k}{k'-h}$. Hence there exist $\ell$ distinct members $G_{1},G_{2},\ldots,G_{\ell}\in \mathcal{G}$ such that $|F_i \cap G_j| \leq t-1$ for all $i,j \in [\ell]$. Then these $\ell$ distinct members in $\mathcal{G}$ satisfy
	$$\sum_{1\leq i,j\leq \ell} |F_i \cap G_j| \leq \ell^{2}(t-1) < \ell^{2}t -\ell+1,$$
	which contradicts the $\ell$-weakly cross $t$-intersecting property of $\mathcal{F}$ and $\mathcal{G}$.

	\paragraph{Case $2$:}  For any $\ell$ members $G_{1}, G_{2}, \ldots, G_{\ell}$ in $\mathcal{G}$, there exists at least one pair $\{i, j\} \subset [\ell]$ such that $|G_{i} \cap G_{j}| \geq t$. We further divide the discussion into two subcases.
	
	\paragraph{Subcase $1$:} $|G_{i}\cap G_{j}|\geq t$ holds for any two sets $G_{i},G_{j}\in \mathcal{G}$. Then we choose arbitrary $G_{1},G_{2},\ldots,G_{\ell}$ such that $|G_{i}\cap G_{j}|= t+m$, where $0\leq m\leq k'-t-1$.
	
	Similar to the proof of Case 1, we obtain $|\mathcal{E}| \leq \ell-1$, $|\mathcal{E}_1| \leq \ell-1$, and
	$$|\mathcal{E}_3| \leq \ell \binom{k'}{t+1}\cdot \binom{n-t-1}{k-t-1}.$$
	It remains to estimate $|\mathcal{E}_2|$.
	
	Note that every member $F$ of the first type satisfies the following property: whenever the size of its intersection with $G_i$ equals $t$, the intersection $F \cap G_i$ must be the same $t$-subset, say $T$. Since $\mathcal{F}$ does not contain an $\mathcal{S}(k,t,r)$ as a subfamily, the number of $k$-subsets $F \in \mathcal{F}$ such that any two of them intersect exactly in $T$ is at most $r-1$.  
	Let $\{F_{1}, F_{2}, \dots, F_{x}\}$ be a maximal collection of such subsets (i.e., $F_{i} \cap F_{j} = T$ for all $i \neq j$). Then any other $F\in \mathcal{E}_2$ of the first type such that $F \cap G_{i}=T$ for some $i$ must intersect at least one of $F_{1}, F_{2}, \dots, F_{x}$ in greater than $t$ points. Since the number of subsets in $\mathcal{E}_2$ of the first type such that $F \cap G_{i} = T$ for some $i$ and $|F\cap F_{j}|>t$ is at most  $\binom{n-t}{k-t}-\binom{n-k}{k-t}$, the number of subsets of the first type in $\mathcal{E}_2$ such that $F \cap G_{i} = T$ for some $i$ is less than $(r-1)\left(\binom{n-t}{k-t}-\binom{n-k}{k-t}+1 \right)$. Notice that $T$ is a subset of the intersection of some two $k'$-sets among $G_{1},\dots,G_{\ell}$, thus the number of such $T$ is at most $\binom{\ell}{2}\binom{k'-1}{t}$. Hence the number of subsets in $\mathcal{E}_2$ of the first type is less than
$$(r-1)\binom{\ell}{2}\binom{k'-1}{t}\left( \binom{n-t}{k-t}-\binom{n-k}{k-t} + 1 \right).$$
	
	For the second type, it is easy to see that if a $k$-set $F$ has intersections with $G_{i}$ and $G_{j}$ that are distinct $t$-subset, then the size of its intersection with $G_{i}\cup G_{j}$ is at least $t+1$. Note that $|G_{i}\cup G_{j}|\leq 2k'$ since $G_{i}$ and $G_{j}$ are $k'$-sets. Therefore, we conclude that $$|\mathcal{E}_{2}|\leq (r-1)\binom{\ell}{2}\binom{k'-1}{t}\left(\binom{n-t}{k-t}-\binom{n-k}{k-t} + 1 \right)+\binom{\ell}{2}\binom{2k'}{t+1} \binom{n-t-1}{k-t-1}.$$

	Since $\mathcal{F}\subseteq \mathcal{E}\cup \mathcal{E}_{1}\cup \mathcal{E}_{2}\cup \mathcal{E}_{3}$, we have 	
	\[
	\begin{aligned}
		|\mathcal{F}| \leq &\; 2\ell-2 + (r-1)\binom{\ell}{2}\binom{k'-1}{t}\left( \binom{n-t}{k-t}-\binom{n-k}{k-t} + 1 \right)\\
		&+ \left( \binom{\ell}{2}\binom{2k'}{t+1} + \ell \binom{k'}{t+1} \right) \binom{n-t-1}{k-t-1}.
	\end{aligned}
	\]
	
	By Theorem~\ref{set-EKR} and Lemma~\ref{setupperbound}, we have  $|\mathcal{G}|\leq \binom{n-t-m}{k'-t-m}$ and 
	\[
	\begin{aligned}
		|\mathcal{F}||\mathcal{G}| 
		&\leq \Bigg[2\ell-2+(r-1)\binom{\ell}{2}\binom{k'-1}{t}\left( \binom{n-t}{k-t}-\binom{n-k}{k-t} + 1 \right) \\
		&+ \left(\binom{\ell}{2}\binom{2k'}{t+1}+\ell \binom{k'}{t+1}\right)\binom{n-t-1}{k-t-1}\Bigg]\binom{n-t-m}{k'-t-m} \\
		&< \binom{n-t}{k'-t}\binom{n-t}{k-t},
	\end{aligned}
	\] 
	which contradicts the initial assumption that the product exceeds this bound.  
	
	\paragraph{Subcase $2$:} There exist two members $G_{1}, G_{2} \in \mathcal{G}$ such that $|G_{1} \cap G_{2}| \leq t-1$. Then we choose $\ell$ members $G_{1}, G_{2}, \ldots, G_{\ell}$ in $\mathcal{G}$ that contain $G_{1}$ and $G_{2}$.

	Similar to the proof of Case 1, we obtain $|\mathcal{E}| \leq \ell-1$, $|\mathcal{E}_1| \leq \ell-1$, and
	$$|\mathcal{E}_3| \leq \ell \binom{k'}{t+1} \binom{n-t-1}{k-t-1}.$$
	It remains to estimate $|\mathcal{E}_2|$.
	
	For the first type, we define a family $$\mathcal{E}_{2}^{I}(T)=\{F\in \mathcal{E}_{2}\colon F\cap G_{j}=T\ \text{whenever}\ |F\cap G_{j}|=t\},$$ for any $t$-subset $T\subset G_{i}$. From the choice of $G_{1}$ and $G_{2}$, we know that at least one of $F\cap G_{1},F\cap G_{2}$ is not equal to $T$. 
	
	Claim: If $\mathcal{E}_{2}^{I}(T)\neq \emptyset$, then $|\mathcal{E}_{2}^{I}(T)|\leq \ell-1$.
	
	Proof of claim: Assume that $\mathcal{E}_{2}^{I}(T)\neq \emptyset$. Since for each $F \in \mathcal{E}_2$ we have $F \cap G_1 \neq T$ or $F \cap G_2 \neq T$, and $|F \cap G_1| \leq t$, $|F \cap G_2| \leq t$. It follows that for any $F \in \mathcal{E}_2^I(T)$, $|F \cap G_1|\leq t-1$ or $|F \cap G_2| \leq t-1$. If $|\mathcal{E}_2^{I}(T)| \geq \ell$, then we may choose any $\ell$ members $F_1, F_2, \ldots, F_\ell$ in $\mathcal{E}_2^{I}(T)$. These $\ell$ members in $\mathcal{E}_2^{I}(T)$ satisfy
	\[
	\sum_{1 \leq i, j \leq \ell}|G_i \cap F_j| \leq \ell(t-1) + (\ell^2 - \ell)t < \ell^2 t - \ell + 1,
	\]
	which contradicts the assumption.

Notice that $T$ must be a subset of the intersection of some two $k'$-sets among $G_1,\dots,G_\ell$ if $\mathcal{E}_{2}^{I}(T)\neq \emptyset$, thus the number of such $T$ is less than $\binom{\ell}{2}\binom{k'-1}{t}$. Hence the number of subsets in $\mathcal{E}_2$ of the first type is less than $\binom{\ell}{2}\binom{k'-1}{t}(\ell-1)$.

	For the second type, it is easy to see that if a $k$-set $F$ has intersections with $G_{i}$ and $G_{j}$ that are distinct $t$-subsets, then the size of its intersection with $G_{i}\cup G_{j}$ is at least $t+1$. Note that $|G_{i}\cup G_{j}|\leq 2k'$ since $G_{i}$ and $G_{j}$ are $k'$-sets. Therefore, we conclude that
	\[
	|\mathcal{E}_{2}|\leq \binom{\ell}{2}\binom{k'-1}{t}(\ell-1)+\binom{\ell}{2}\binom{2k'}{t+1} \binom{n-t-1}{k-t-1}.
	\]
	
	Since $\mathcal{F}\subseteq \mathcal{E}\cup \mathcal{E}_{1}\cup \mathcal{E}_{2}\cup \mathcal{E}_{3}$, we have 
	\[	
	|\mathcal{F}|\leq \left(\binom{\ell}{2}\binom{k'-1}{t}+2\right)(\ell-1)+\left(\binom{\ell}{2}\binom{2k'}{t+1}+\ell\binom{k'}{t+1}\right)\binom{n-t-1}{k-t-1}.
	\]
	By assumption $|\mathcal{F}|\cdot|\mathcal{G}|\geq\binom{n-t}{k-t} \binom{n-t}{k'-t}$ and Lemma~\ref{setupperbound2}, we have  
	\[
	|\mathcal{G}| \geq \ell\sum_{h=t}^{k'}\binom{k}{h}\binom{n-k}{k'-h}+\ell.
	\]
	a similar argument as in Case 1 yields the same contradiction.
	
	In summary, we have completed the proof.
	\qed
\end{proof}

If $\mathcal{G}$ does not contain a sunflower $\mathcal{S}(k',t,r')$, a similar argument shows that $|\mathcal{F}||\mathcal{G}|< \binom{n-t}{k-t}\binom{n-t}{k'-t}$, and thus the first part of Theorem~\ref{Set} is proved.

By Lemma~\ref{Nsunflowerset}, if $|\mathcal{F}||\mathcal{G}| = \binom{n-t}{k'-t}\binom{n-t}{k-t}$, then $\mathcal{F}$ and $\mathcal{G}$ must contain a sunflower $\mathcal{S}(k,t,r)$ and a sunflower $\mathcal{S}(k',t,r')$, respectively. Then, by Lemma~\ref{setsunflower}, every member $G\in\mathcal{G}$ contains the kernel $K$ of $\mathcal{S}(k,t,r)$ and every member $F\in \mathcal{F}$ contains the kernel $K'$ of $\mathcal{S}(k',t,r')$. If $K\neq K'$, since $n\geq  \frac{k^{2}\ell^{4}}{2}\binom{2k}{t+1}\binom{k}{t}+t$, we can choose $\ell$ subsets $F_{1},\ldots,F_{\ell}\in \mathcal{F}$ and  $\ell$ subsets $G_{1},\ldots,G_{\ell}\in \mathcal{G}$ such that $|F_{i}\cap G_{j}|\leq t-1$ holds for every $i,j\in [\ell]$. Then $\sum_{1 \leq i, j \leq \ell}|G_{i} \cap F_{j}| \leq \ell^2(t-1) < \ell^2 t-\ell+1,$ which contradicts the assumption that $\mathcal{F}$ and $\mathcal{G}$ are $\ell$-weakly cross $t$-intersecting families. Thus, $K=K'$ and the Theorem~\ref{Set} is proved.

\section{Proof of Theorem \ref{Main}}

For a positive integer $n$, let $V$ be an $n$ dimensional vector space over $\mathbb{F}_{q}$. Similar to the set version, the proof of the subspace version also utilizes the sunflower method.  A $k$-uniform sunflower is a family of $k$-dimensional subspaces in which all pairs of distinct subspaces have the same intersection (called the kernel, the remaining parts are called petals).  For positive integers $k, t, u$ with $t\leq k\leq n$, we denote by $\mathcal{S}_{q}(k,t,u)\subseteq \genfrac{[}{]}{0pt}{}{V}{k}$ a $k$-uniform sunflower with a kernel of dimension $t$ and $u$ petals.

For the proof of Theorem~\ref{Main}, we first need the following lemmas.

Since the structure of a subspace differs from that of a set, we need to prove the subspace version of Lemma~\ref{setsunflower}.

\begin{Lemma}\label{sunflower}
	Let $r=(\genfrac{[}{]}{0pt}{}{k'}{1}+1)\ell$. Let $\mathcal {F}\subseteq \genfrac{[}{]}{0pt}{}{V}{k}$ and $\mathcal{G}\subseteq \genfrac{[}{]}{0pt}{}{V}{k'}$ be $\ell$-weakly cross $t$-intersecting. Suppose that $\mathcal{F}$ contains a sunflower $\mathcal{S}_{q}(k,t,r)$ with the kernel $K$ as a subfamily. Then every member of $\mathcal{G}$ contains $K$, i.e., $|\mathcal{G}|\leq \genfrac{[}{]}{0pt}{}{n-t}{k'-t}.$
\end{Lemma}

\begin{proof} We proceed by contradiction. Suppose that there exist a subspace $G_{\ell}\in \mathcal{G}$ that does not contain $K$. We choose $\ell$ distinct subspaces from $\mathcal{G}$ as follows. Denote by $\mathcal{G}(K)$ the set of all members of $\mathcal{G}$ that contain $K$. If $|\mathcal{G}(K)|=d\geq \ell-1$, then we choose $G_{\ell}$ and $\ell-1$ distinct subspaces $G_1,\dots,G_{\ell-1}\in\mathcal{G}(K)$. Otherwise, we choose $G_1,\dots,G_d\in \mathcal{G}(K)$ and $\ell-d$ subspaces $G_{d+1},\dots,G_\ell\in\mathcal{G}\backslash \mathcal{G}(K)$.
	
	Let $h = \min\{d, \ell-1\}$ and let $\mathcal{S}_0$ be the sunflower $\mathcal{S}_{q}(k,t,r)$ in $\mathcal{F}$. For each $i\in [h]$, since the intersection of the petals of a sunflower is trivial, there are at most $\genfrac{[}{]}{0pt}{}{k'}{1}-\genfrac{[}{]}{0pt}{}{t}{1}$ members of $\mathcal{S}_0$ that have a nonempty intersection with $G_i\setminus K$. Thus, after deleting such subspaces, we obtain a new sunflower $\mathcal{S}_1=\mathcal{S}_{q}(k,t,r')\subseteq \mathcal{S}_0$ of size at least $r'=r -h(\genfrac{[}{]}{0pt}{}{k'}{1}-\genfrac{[}{]}{0pt}{}{t}{1})$, where the intersection of each member with $G_i$ is exactly $K$. Similarly, for each $j\in \{h+1,\dots,\ell\}$, there are at most $\genfrac{[}{]}{0pt}{}{k'}{1}$ members of $\mathcal{S}_1$ that have a nonempty intersection with $G_j\setminus K$. Thus, after deleting such subspaces from $\mathcal{S}_1$, we obtain another sunflower $\mathcal{S}_2=\mathcal{S}_{q}(k,t,r'')\subseteq \mathcal{S}_1$ of size at least $r''=r-h(\genfrac{[}{]}{0pt}{}{k'}{1}-\genfrac{[}{]}{0pt}{}{t}{1})-(\ell - h)\genfrac{[}{]}{0pt}{}{k'}{1}\geq \ell$ such that for each $i\in [\ell]$, the intersection of each member of $\mathcal{S}_2$ with $G_i$ is a subspace of $K$. Since $G_{\ell}$ does not contain kernel $K$, $\dim(F\cap G_\ell) \leq t-1$ holds for each $F\in \mathcal{S}_2$. Thus, if we take $\ell$ distinct elements $F_1,\dots,F_\ell\in \mathcal{S}_2$, then $$\sum_{1\leq i,j\leq \ell}\dim(F_i\cap G_j) \leq \ell(\ell-1)t + \ell(t-1) < \ell^2 t - \ell+1,$$
which contradicts the $\ell$-weakly cross $t$-intersecting property of  $\mathcal{F}$ and $\mathcal{G}$.
 
Hence, each member of $\mathcal{G}$ contains $K$. The conclusion follows.\qed
\end{proof}

Since the roles of $\mathcal{F}$ and $\mathcal{F}'$ are symmetric, we obtain the following corollary as an immediate consequence of Lemma \ref{sunflower}.

\begin{Corollary}\label{Cor2}
	Let $r = (\genfrac{[}{]}{0pt}{}{k'}{1}+1)\ell$ and $r'=(\genfrac{[}{]}{0pt}{}{k}{1}+1)\ell$. Let $\mathcal {F}\subseteq \genfrac{[}{]}{0pt}{}{V}{k}$ and $\mathcal{G}\subseteq \genfrac{[}{]}{0pt}{}{V}{k'}$ be $\ell$-weakly cross $t$-intersecting. Suppose that $\mathcal{F}$ contains a $\mathcal{S}_{q}(k,t,r)$ as a subfamily and that $\mathcal G$ contains a $\mathcal{S}_q(k',t,r')$ as a subfamily. Then $|\mathcal{F}|\cdot |\mathcal{G}| \leq \genfrac{[}{]}{0pt}{}{n-t}{k-t}\genfrac{[}{]}{0pt}{}{n-t}{k'-t}$.
\end{Corollary}

 The conclusion of Corollary \ref{Cor2} attains the upper bound of Theorem~\ref{Main}. Therefore, we can assume that either $\mathcal{F}$ does not contain $\mathcal{S}_{q}(k,t,r)$ or $G$ does not contain $\mathcal{S}_{q}(k',t,r')$.

\begin{Lemma}\label{Nsunflower}
	Let $\mathcal {F}\subseteq \genfrac{[}{]}{0pt}{}{V}{k}$ and $\mathcal{G}\subseteq \genfrac{[}{]}{0pt}{}{V}{k'}$ be $\ell$-weakly cross $t$-intersecting. Suppose that $\mathcal{F}$ does not contain $\mathcal{S}_{q}(k,t,r)$ as a subfamily, where $r = (\genfrac{[}{]}{0pt}{}{k'}{1}+1)\ell$. Then $|\mathcal{F}||\mathcal{G}|< \genfrac{[}{]}{0pt}{}{n-t}{k-t}\genfrac{[}{]}{0pt}{}{n-t}{k'-t}$.
\end{Lemma}

\begin{proof} We proceed by contradiction. Assume  $|\mathcal{F}|\cdot |\mathcal{G}|\geq \genfrac{[}{]}{0pt}{}{n-t}{k-t}\genfrac{[}{]}{0pt}{}{n-t}{k'-t},$ and we split the proof into two cases. 
	
	\paragraph{Case $1$:} There exist $\ell$ subspaces $G_{1},G_{2},\ldots,G_{\ell}$ in $\mathcal{G}$ such that $\dim(G_{i}\cap G_{j})\leq t-1$ for any $i\neq j\in [\ell]$. Then we choose such $\ell$ subspaces.
	
	By letting $\mathcal{E} = \{F\in \mathcal{F}\colon \dim(G_{i}\cap F)\leq t -1 \text{ for every } i\in [\ell]\}$, we claim that $|\mathcal{E}|\leq \ell-1$. Otherwise, if there exist $\ell$ members $F_{1},F_{2},\ldots,F_{\ell}\in \mathcal{F}$ such that $\dim(G_{i}\cap F_{j})\leq t-1$ for every $i,j\in [\ell]$, then
$$\sum_{1\leq i,j\leq \ell}\dim(G_{i}\cap F_{j})\leq \ell^{2}(t-1)<\ell^{2}t-\ell+1,$$ which contradicts the assumption that $\mathcal{F}$ and $\mathcal{G}$ are $\ell$-weakly cross $t$-intersecting.
	
	By the definition of $\mathcal{E}$, for every element $F$ in $\mathcal{F}\setminus \mathcal{E}$, there exists at least one $i \in [\ell]$ such that $\dim(G_{i}\cap F)\geq t$. 
	Let $$\mathcal{E}_{1}=\{F\in \mathcal{F}\colon \dim(G_{i}\cap F)=t \text{ for some } i\in [\ell],\dim(G_{j}\cap F)\leq t - 1 \text{ for each } j\in [\ell]\backslash \{i\} \},$$
	$$\mathcal{E}_{2} = \bigl\{ F \in \mathcal{F} \colon \dim(G_{i} \cap F) \le t \text{ for all } i \in [\ell], \text{ and } \bigl| \{ i \in [\ell] : \dim(G_{i} \cap F) = t \} \bigr| \ge 2 \bigr\},$$
	$$\mathcal{E}_{3}=\{F\in \mathcal{F}\colon \dim(G_{i}\cap F)\geq t+1 \text{ for some } i\in [\ell]\}.$$
	
	Claim: $|\mathcal{E}_{1}|\leq \ell-1$.
	
	Proof of claim: We proceed by contradiction. Assume $|\mathcal{E}_{1}|\geq \ell$ and choose $\ell$ distinct elements $F_{1},F_{2},\ldots,F_{\ell}\in\mathcal{E}_{1}$. Then $$\sum_{1\leq i,j\leq \ell}\dim(G_{i}\cap F_{j})\leq \ell(\ell-1)(t-1)+\ell t<\ell^{2}t-\ell+1,$$
which contradicts the assumption that $\mathcal{F}$ and $\mathcal{G}$ are $\ell$-weakly cross $t$-intersecting. 
	
	Next, we estimate the sizes of $\mathcal{E}_{2} $ and $\mathcal{E}_{3}$.
	
	The subspaces in $\mathcal{E}_{2}$ can be classified into two classes. The first class consists of all $k$-dimensional subspaces $F$ whose intersections with those $G_{i}$ among $G_{1}, G_{2}, \ldots, G_{\ell}$ that satisfy $\dim(F\cap G_{i}) = t$ are the same $t$-dimensional subspace, say $T$. The second class consists of all $k$-dimensional subspaces whose intersections with two subspaces $G_{i}$ and $G_{j}$ among $G_{1}, G_{2}, \ldots, G_{\ell}$ are two distinct $t$-dimensional subspaces. For the first type, the number of such subspaces in $\mathcal{E}_2$ is $0$ (since $F\cap G_{i} = F\cap G_{j}=T$ implies $T\subseteq G_{i}\cap G_{j}$, contradicting the choice of $G_{1},G_{2},\ldots,G_{\ell}$). 
	For the second type, it is easy to see that if a $k$-dimensional subspace $F$ has intersections with $G_{i}$ and $G_{j}$ that are distinct $t$-dimensional subspaces, then the dimension of its intersection with $G_{i}+G_{j}$ is at least $t+1$. Note that $\dim(G_{i}+G_{j})\leq 2k'$ since $G_{i}$ and $G_{j}$ are $k'$-dimensional subspaces. Therefore, we conclude that 
	$$\mathcal{E}_{2}\subseteq \bigcup_{1\leq i<j\leq \ell}(\bigcup_{A\subseteq G_{i}+G_{j},\dim(A)=t+1}\mathcal{F}(A)),$$
	where $\mathcal{F}(A)=\{F\in \genfrac{[}{]}{0pt}{}{V}{k}\colon A\subset F\}$. By Lemma \ref{Spacenumber}, we have
      $$|\mathcal{E}_{2}|\leq \binom{\ell}{2}\genfrac{[}{]}{0pt}{}{2k'}{t+1}\cdot \genfrac{[}{]}{0pt}{}{n-t-1}{k-t-1}.$$
	It is easy to see that $$\mathcal{E}_{3}\subseteq \bigcup_{i=1}^{\ell}(\bigcup_{A\subseteq G_{i},\dim(A)=t+1}\mathcal{F}(A)).$$
	Therefore, we conclude that 
       $$|\mathcal{E}_{3}|\leq \ell \genfrac{[}{]}{0pt}{}{k'}{t+1}\cdot \genfrac{[}{]}{0pt}{}{n-t-1}{k-t-1}.$$
	
	Since $\mathcal{F}\subseteq \mathcal{E}\cup \mathcal{E}_{1}\cup \mathcal{E}_{2}\cup \mathcal{E}_{3}$, we have 
	$$|\mathcal{F}|\leq 2\ell-2+\left[\binom{\ell}{2}\genfrac{[}{]}{0pt}{}{2k'}{t+1}+\ell \genfrac{[}{]}{0pt}{}{k'}{t+1}\right]\genfrac{[}{]}{0pt}{}{n-t-1}{k-t-1}.$$
	By  assumption $|\mathcal{F}|\cdot |\mathcal{G}| \geq \genfrac{[}{]}{0pt}{}{n-t}{k-t}\genfrac{[}{]}{0pt}{}{n-t}{k'-t}$ and Lemma \ref{Upperbound2},
	we have  $$|\mathcal{G}| \geq \ell\sum_{h=t}^{k'} q^{(k-h)(k'-h)} \genfrac{[}{]}{0pt}{}{k}{h}\genfrac{[}{]}{0pt}{}{n-k}{k'-h}+\ell.$$
	
	To reach a contradiction, pick $\ell$ distinct subspaces $F_{1},F_{2},\ldots,F_{\ell}$ of $\mathcal{F}$.  
	For each fixed $j \in [\ell]$, the number of $k'$-dimensional subspaces intersecting $F_j$ in dimension at least $t$ is at most  
	$\sum_{h=t}^{k'} q^{(k-h)(k'-h)} \genfrac{[}{]}{0pt}{}{k}{h}\genfrac{[}{]}{0pt}{}{n-k}{k'-h}$ by Lemma \ref{Spacenumber}. Hence there exist $\ell$ distinct members $G_{1},G_{2},\ldots,G_{\ell}\in \mathcal{G}$ such that $\dim(F_i \cap G_j) \le t-1$ for all $i,j \in [\ell]$. Then these $\ell$ members in $\mathcal{G}$ satisfy $$\sum_{1\leq i,j\leq \ell} \dim(F_i \cap G_j) \leq \ell^{2}(t-1) < \ell^{2}t -\ell+1,$$ 
	which contradicts the $\ell$-weakly cross $t$-intersecting property of $\mathcal{F}$ and $\mathcal{G}$.
	
	\paragraph{Case $2$:}  For any $\ell$ subspaces $G_{1}, G_{2}, \ldots, G_{\ell}$ in $\mathcal{G}$, there exists at least one pair $\{i, j\}\subset [\ell]$ such that $\dim(G_{i} \cap G_{j}) \geq t$. We further divide the discussion into two subcases.
	
	\paragraph{Subcase $1$:}  $\dim(G_{i}\cap G_{j})\geq t$ for any two subspaces $G_{i},G_{j}\in \mathcal{G}$. Then we choose arbitrary $G_{1},G_{2},\ldots,G_{\ell}$ such that $\dim(G_{i}\cap G_{j})\geq t+m$ for all $i,j\in [\ell]$, where $0\leq m\leq k'-t-1$.
	
	Similar to the proof of Case 1, we obtain $|\mathcal{E}| \leq \ell-1$, $|\mathcal{E}_{1}| \leq \ell-1$, and
       $$|\mathcal{E}_3| \leq \ell \genfrac{[}{]}{0pt}{}{k'}{t+1}\cdot \genfrac{[}{]}{0pt}{}{n-t-1}{k-t-1}.$$
	It remains to estimate $|\mathcal{E}_{2}|$.

	Note that every subspace $F$ of the first type satisfies the following property: whenever the dimension of its intersection with $G_i$ equals $t$, then the intersection $F \cap G_i$ must be the $t$-subspace, say $T$. Since $\mathcal{F}$ does not contain a sunflower $\mathcal{S}_{q}(k,t,r)$ as a subfamily, the number of $k$-subspaces $F \in \mathcal{F}$ such that any two of them intersect exactly in $T$ is at most $r-1$.  
	Let $\{F_{1}, F_{2}, \dots, F_{x}\}$ be a maximal collection of such subspaces (i.e., $F_{i} \cap F_{j} = T$ for all $i \neq j$). Then each $F$ of other subspaces of the first-type satisfying $F \cap G_{i} = T$ for some $i$ must intersect at least one of $F_{1},F_{2},\dots,F_{x}$ in a subspace with dimension greater than $t$. By Lemma~\ref{Spacenumber}, the number of subspaces in $\mathcal{E}_{2}$ of the first type such that $F \cap G_{i}=T$ for some $i$ and $\dim (F\cap F_{j})>t$ is at most  $\genfrac{[}{]}{0pt}{}{n-t}{k-t}-q^{(k-t)^2}\genfrac{[}{]}{0pt}{}{n-k}{k-t}$, thus the number of subspaces in $\mathcal{E}_{2}$ of  first-type such that $F \cap G_{i} = T$ for some $i$ is less than $(r-1)\left( \genfrac{[}{]}{0pt}{}{n-t}{k-t}-q^{(k-t)^2}\genfrac{[}{]}{0pt}{}{n-k}{k-t}+1 \right)$. Notice that $T$ is a subspace of the intersection of some two subspaces among $G_{1},\dots,G_{\ell}$, thus the number of such $T$ is at most $\binom{\ell}{2}\genfrac{[}{]}{0pt}{}{k'-1}{t}$. Hence the number of  subspaces in $\mathcal{E}_{2}$ of the first-type is less than
$$(r-1)\binom{\ell}{2}\genfrac{[}{]}{0pt}{}{k'-1}{t}\left( \genfrac{[}{]}{0pt}{}{n-t}{k-t} - q^{(k-t)^2}\genfrac{[}{]}{0pt}{}{n-k}{k-t} + 1 \right).$$
	
	For the second type, it is easy to see that if a $k$-dimensional subspace $F$ has intersections with $G_{i}$ and $G_{j}$ that are distinct $t$-dimensional subspaces, then the dimension of its intersection with $G_{i}+G_{j}$ is at least $t+1$. Note that $\dim(G_{i}+G_{j})\leq 2k'$ since $G_{i}$ and $G_{j}$ are $k'$-dimensional subspaces. Therefore, we conclude that
	$$|\mathcal{E}_{2}|\leq (r-1)\binom{\ell}{2}\genfrac{[}{]}{0pt}{}{k'-1}{t}\left(\genfrac{[}{]}{0pt}{}{n-t}{k-t} - q^{(k-t)^{2}}\genfrac{[}{]}{0pt}{}{n-k}{k-t}+1\right)+\binom{\ell}{2}\genfrac{[}{]}{0pt}{}{2k'}{t+1}\cdot \genfrac{[}{]}{0pt}{}{n-t-1}{k-t-1}.$$
	
	Since $\mathcal{F}\subseteq \mathcal{E}\cup \mathcal{E}_{1}\cup \mathcal{E}_{2}\cup \mathcal{E}_{3}$, we have 
	
	\[
	\begin{aligned}
		|\mathcal{F}| &\leq 2\ell-2+(r-1)\binom{\ell}{2}\genfrac{[}{]}{0pt}{}{k'-1}{t}\left(\genfrac{[}{]}{0pt}{}{n-t}{k-t} - q^{(k-t)^{2}}\genfrac{[}{]}{0pt}{}{n-k}{k-t}+1\right)\\
		&+\Bigg( \binom{\ell}{2}\genfrac{[}{]}{0pt}{}{2k'}{t+1}+ \ell \genfrac{[}{]}{0pt}{}{k'}{t+1} \Bigg) \genfrac{[}{]}{0pt}{}{n-t-1}{k-t-1}.
	\end{aligned}
	\]
	
	By Theorem \ref{Cross} and Lemma \ref{Upperbound}, we have  $|\mathcal{G}|\leq \genfrac{[}{]}{0pt}{}{n-t-m}{k'-t-m}$ and 
	\[
	\begin{aligned}
		|\mathcal{F}||\mathcal{G}| 
		&\leq \Bigg[2\ell-2+ (r-1)\binom{\ell}{2}\genfrac{[}{]}{0pt}{}{k'-1}{t}\left(\genfrac{[}{]}{0pt}{}{n-t}{k-t} - q^{(k-t)^{2}}\genfrac{[}{]}{0pt}{}{n-k}{k-t}+1\right)\\
		&+ \left(\binom{\ell}{2}\genfrac{[}{]}{0pt}{}{2k'}{t+1}+\ell \genfrac{[}{]}{0pt}{}{k'}{t+1}\right)\genfrac{[}{]}{0pt}{}{n-t-1}{k-t-1}\Bigg]\genfrac{[}{]}{0pt}{}{n-t-m}{k'-t-m} \\
		&< \genfrac{[}{]}{0pt}{}{n-t}{k'-t}\genfrac{[}{]}{0pt}{}{n-t}{k-t},
	\end{aligned}
	\] 
	which contradicts the initial assumption that the product exceeds this bound.  
	
	\paragraph{Subcase $2$:} There exist two subspaces $G_{1}, G_{2} \in \mathcal{G}$ such that $\dim(G_{1} \cap G_{2}) \leq t-1$. Then we choose $\ell$ subspaces $G_{1}, G_{2}, \ldots, G_{\ell}$ in $\mathcal{G}$ that contain $G_{1}$ and $G_{2}$.

	Similar to the proof of Case 1, we obtain $|\mathcal{E}| \leq \ell-1$, $|\mathcal{E}_1| \leq \ell-1$, and
	$$|\mathcal{E}_3| \leq \ell \genfrac{[}{]}{0pt}{}{k'}{t+1}\cdot \genfrac{[}{]}{0pt}{}{n-t-1}{k-t-1}.$$

	It remains to estimate $|\mathcal{E}_{2}|$.
	
	For the first type, we define a family $$\mathcal{E}_{2}^{I}(T)=\{F\in \mathcal{E}_{2}\colon F\cap G_{i}=T \text{ whenever } \dim(F\cap G_{i})=t\},$$ for any $t$-dimensional subspace $T$ of $G_{j}$. From the choice of $G_{1}$ and $G_{2}$, we know that at least one of $F\cap G_{1},F\cap G_{2}$ is not equal to $T$. 
	
	Claim: If $\mathcal{E}_{2}^{I}(T)\neq \emptyset$, then $|\mathcal{E}_{2}^{I}(T)|\leq \ell-1$.
	
	Proof of claim: Suppose that $\mathcal{E}_{2}^{I}(T)\neq \emptyset$. Since for each $F \in \mathcal{E}_2$ we have $F \cap G_{1} \neq T$ or $F \cap G_{2} \neq T$, and $\dim(F \cap G_{1}) \leq t$, $\dim(F \cap G_{2}) \leq t$, it follows that for any $F \in \mathcal{E}_2^I(T)$, $\dim(F \cap G_{1}) \leq t-1$ or $\dim(F \cap G_{2}) \leq t-1$. If $|\mathcal{E}_2^{I}(T)| \geq \ell$, then we may choose any $\ell$ subspaces $F_{1}, F_{2}, \ldots, F_\ell$ in $\mathcal{E}_2^{I}(T)$.  Then these $\ell$ members in $\mathcal{E}_2^{I}(T)$ satisfy
$$\sum_{1 \leq i, j \leq \ell} \dim(G_{i} \cap F_{j}) \leq \ell(t-1)+(\ell^2-\ell)t < \ell^2 t-\ell+1,$$
which contradicts the assumption.
	
	Notice that $T$ must be a subspace of the intersection of some two $k'$-subspaces among $G_1,\dots,G_\ell$ if $\mathcal{E}_{2}^{I}(T)\neq \emptyset$, thus the number of such $T$ is less than $\binom{\ell}{2}\genfrac{[}{]}{0pt}{}{k'-1}{t}$. Hence the number of subspaces in $\mathcal{E}_2$ of the first type is less than $\binom{\ell}{2}\genfrac{[}{]}{0pt}{}{k'-1}{t}(\ell-1)$.
	
	For the second type, it is easy to see that if a $k$-dimensional subspace $F$ has intersections with $G_{i}$ and $G_{j}$ that are distinct $t$-dimensional subspaces, then the dimension of its intersection with $G_{i}+G_{j}$ is at least $t+1$. Note that $\dim(G_{i}+G_{j})\leq 2k'$ since $G_{i}$ and $G_{j}$ are $k'$-dimensional subspaces. Therefore, we conclude that
	$$|\mathcal{E}_{2}|\leq \binom{\ell}{2}\genfrac{[}{]}{0pt}{}{k'-1}{t}(\ell-1)+\binom{\ell}{2}\genfrac{[}{]}{0pt}{}{2k'}{t+1}\cdot \genfrac{[}{]}{0pt}{}{n-t-1}{k-t-1}.$$
	
	Since $\mathcal{F}\subseteq \mathcal{E}\cup \mathcal{E}_{1}\cup \mathcal{E}_{2}\cup \mathcal{E}_{3}$, we have 
	$$|\mathcal{F}|\leq  \left(\binom{\ell}{2}\genfrac{[}{]}{0pt}{}{k'-1}{t}+2\right)(\ell-1)+\left(\binom{\ell}{2}\genfrac{[}{]}{0pt}{}{2k'}{t+1}+\ell\genfrac{[}{]}{0pt}{}{k'}{t+1}\right)\genfrac{[}{]}{0pt}{}{n-t-1}{k-t-1}.$$
	By assumption $|\mathcal{F}|\cdot|\mathcal{G}|\geq \genfrac{[}{]}{0pt}{}{n-t}{k-t} \genfrac{[}{]}{0pt}{}{n-t}{k'-t}$ and Lemma \ref{Upperbound2}, we have  $$|\mathcal{G}| \geq \ell\sum_{h=t}^{k'} q^{(k-h)(k'-h)} \genfrac{[}{]}{0pt}{}{k}{h}\genfrac{[}{]}{0pt}{}{n-k}{k'-h}+\ell,$$
a similar argument as in Case 1 yields the same contradiction.
	
	In summary, we have completed the proof.
	\qed
\end{proof}

If $\mathcal{G}$ does not contain a sunflower $\mathcal{S}_{q}(k',t,r')$, a similar argument shows that $|\mathcal{F}||\mathcal{G}|< \genfrac{[}{]}{0pt}{}{n-t}{k-t}\genfrac{[}{]}{0pt}{}{n-t}{k'-t}$, and thus the first part of Theorem~\ref{Main} is proved.

By Lemma~\ref{Nsunflower}, if $|\mathcal{F}||\mathcal{G}|= \genfrac{[}{]}{0pt}{}{n-t}{k'-t}\genfrac{[}{]}{0pt}{}{n-t}{k-t}$, then $\mathcal{F}$ and $\mathcal{G}$ must contain a sunflower $\mathcal{S}_{q}(k,t,r)$ and a sunflower $\mathcal{S}_{q}(k',t,r')$, respectively.  Then, by Lemma~\ref{sunflower}, every subspace in $\mathcal{G}$ contains the kernel $K$ of $\mathcal{S}_{q}(k,t,r)$ and every subspace in $\mathcal{F}$ contains the kernel $K'$ of $\mathcal{S}_{q}(k',t,r')$. If $K\neq K'$, since $n\geq  (2k-t+1)(t+1)+(k-t+1)k'+k+2\ell-1$, we can choose $\ell$ subspaces $F_{1},\ldots,F_{\ell}\in \mathcal{F}$ and  $\ell$ subspaces $G_{1},\ldots,G_{\ell}\in \mathcal{G}$ such that $\dim(F_{i}\cap G_{j})\leq t-1$ holds for every $i,j\in [\ell]$. Then $\sum_{1 \leq i, j \leq \ell} \dim(G_{i} \cap F_{j}) \leq \ell^2(t-1) < \ell^2 t-\ell+1,$ which contradicts the assumption that $\mathcal{F}$ and $\mathcal{G}$ are $\ell$-weakly cross $t$-intersecting families. Thus, $K=K'$ and the Theorem~\ref{Main} is proved.

\section{Conclusion}

In this paper, we provided an alternative proof of the set version of the $\ell$-weakly cross $t$-intersecting theorem and determined an explicit lower bound for $n$ with some constraints on parameters $k,k',t$ and $\ell$. We also gave a subspace version of $\ell$-weakly cross $t$-intersecting theorem.  Theorem \ref{Cross} was further improved by Wen and Lv who proved that lower bound for $n$ does not depend on $t$ \cite{Wen2026}. It is worthwhile to determine a sharper lower bound for $n$ for general $k, k', t$ and $\ell$.  

\section*{Acknowledgement}
During the preparation of this work, the first author used DeepSeek as a tool for only grammar correction and moderate language refinement. After using this tool, the authors reviewed and edited the content as needed and take full responsibility for the content of the published article.
The work was supported by the National Natural Science Foundation of China under grants No. 12501463 (S. Yu) and No. 12271390 (L. Ji).

\end{document}